\documentclass[12pt]{amsart}
\usepackage{a4wide}
\usepackage[utf8]{inputenc}
\usepackage{amsmath}
\usepackage{amssymb}
\usepackage{amsfonts}
\usepackage{amsopn}
\usepackage{graphicx}
\usepackage{enumerate}
\usepackage{color}
\usepackage{mathtools}
\usepackage{mathrsfs}
\usepackage{bbm}
\usepackage{yhmath}
\usepackage{cite}
\usepackage[colorlinks,linkcolor=blue,anchorcolor=blue,citecolor=blue]{hyperref}

\newtheorem{theorem}{Theorem}[section]
\newtheorem{proposition}[theorem]{Proposition}
\newtheorem{lemma}[theorem]{Lemma}
\newtheorem{definition}[theorem]{Definition}
\newtheorem{corollary}[theorem]{Corollary}
\theoremstyle{definition}
\newtheorem{example}[theorem]{Example}
\theoremstyle{remark}

\numberwithin{equation}{section}

\newcommand{\mcal}{\mathcal}

\newcommand{\set}[1]{\left\{ #1 \right\}}

\newcommand{\C}{\mathbb{C}}
\newcommand{\R}{\mathbb{R}}

\newcommand{\Z}{\mathbb{Z}}
\newcommand{\N}{\mathbb{N}}
\newcommand{\PP}{\mathbb{P}}

\newcommand{\f}{\infty}
\newcommand{\la}{\langle}
\newcommand{\ra}{\rangle}
\newcommand{\wh}[1]{\widehat{#1}}

\newcommand{\ep}{\varepsilon}
\newcommand{\sm}{\setminus}
\newcommand{\D}{\;\mathrm{d}}

\title[Spectrality of Infinite Convolutions and Random Convolutions]{Spectrality of Infinite Convolutions and Random Convolutions}
\author[W. Li]{Wenxia Li}
\address[W. Li]{School of Mathematical Sciences, Key Laboratory of MEA (Ministry of Education) \& Shanghai Key Laboratory of PMMP, East China Normal University, Shanghai 200241, People's Republic of China}
\email{wxli@math.ecnu.edu.cn}

\author[J. J. Miao]{Jun Jie Miao}
\address[J. J. Miao]{School of Mathematical Sciences, Key Laboratory of MEA (Ministry of Education) \& Shanghai Key Laboratory of PMMP, East China Normal University, Shanghai 200241, People's Republic of China}
\email{jjmiao@math.ecnu.edu.cn}

\author[Z. Wang]{Zhiqiang Wang*}
\address[Z. Wang]{School of Mathematical Sciences, Key Laboratory of MEA (Ministry of Education) \& Shanghai Key Laboratory of PMMP, East China Normal University, Shanghai 200241,
People's Republic of China}
\email{zhiqiangwzy@163.com}

\thanks{* Corresponding author}

\begin{document}
\begin{abstract}
In this paper, we explore spectral measures whose square integrable spaces admit a family of exponential functions as an orthonormal basis.
Our approach involves utilizing the integral periodic zeros set of Fourier transform to characterize spectrality of infinite convolutions generated by a sequence of admissible pairs.
Then we delve into the analysis of the integral periodic zeros set.
Finally, we show that given finitely many admissible pairs, almost all random convolutions are spectral measures.
Moreover, we give a complete characterization of spectrality of random convolutions in some special cases.
\end{abstract}

\maketitle

\section{Introduction}

\subsection{Spectral measures and infinite convolutions}
In 1974, Fuglede \cite{Fuglede-1974} studied the existence of commuting self-adjoint partial differential operators and proposed the well-known spectral set conjecture.
\begin{quote}
  \textbf{Spectral Set Conjecture}: \emph{Let $\Gamma \subset \R^d$ be a measurable set with positive finite Lebesgue measure. Then there exists a set $\Lambda \subset \R^d$ such that $\{ e_{\lambda}(x) = e^{2\pi i \lambda \cdot x}: \lambda \in \Lambda\}$ forms an orthogonal basis for $L^2(\Gamma)$,
  if and only if $\Gamma$ tiles $\R^d$ by translations.}
\end{quote}
The spectral set conjecture connects the analysis and geometry on a set.
It remains open until Tao \cite{Tao-2004} gave the first counterexample in higher dimensions $d \ge 5$ in 2004.
Later, some counterexamples in lower dimensions were also constructed \cite{Farkas-Revesz-2006, Farkas-Matolcsi-Mora-2006,Kolountzakis-Matolcsi-2006a,Kolountzakis-Matolcsi-2006b,Matolcsi-2005}.
The conjecture is still open in dimensions $d=1$ and $d=2$.
Recently, Nev and Matolcsi \cite{Lev-Matolcsi-2019} showed that the spectral set conjecture holds in all dimensions for general convex bodies (that is, a compact convex set with non-empty interior).

A Borel probability measure $\mu$ on $\R^d$ is called a \emph{spectral measure} if there exists a set $\Lambda \subset \R^d$ such that the family of exponential functions
$$E(\Lambda)= \big\{ e_\lambda(x) = e^{2\pi i \lambda \cdot x} : \lambda \in \Lambda \big\}$$
forms an orthonormal basis for $L^2(\mu)$, where the set $\Lambda$ is called a \emph{spectrum} of $\mu$.
Note that $E(\Lambda)$ is an orthonormal basis for $L^2(\mu)$ if and only if
\begin{itemize}
  \item orthogonality: for all $\lambda \ne \lambda' \in \Lambda$,
  $$\la e_\lambda, e_{\lambda'} \ra_{L^2(\mu)} = \int_{\R^d} e^{2\pi i (\lambda - \lambda') \cdot x} \D \mu(x) = \wh{\mu}(\lambda'-\lambda)=0,$$ where $\wh{\mu}(\xi)$ is the Fourier transform of $\mu$;
  \item completeness: if $f \in L^2(\mu)$, and $\la f, e_{\lambda} \ra_{L^2(\mu)}=0$ for all $\lambda \in \Lambda$, then $f=0$ $\mu$-a.e.
\end{itemize}
In the context of classical Fourier analysis, the Lebesgue measure on the unit hypercube $[0,1]^d$ is a spectral measure with a spectrum $\Z^d$.
The support of a spectral measure typically exhibits a strong geometric structure, which is generally uncommon in nature.

In 1998, Jorgensen and Pedersen \cite{Jorgensen-Pedersen-1998} discovered that the self-similar measure $\mu_{4,\{0,2\}}$ satisfying the equation
$$\mu(\;\cdot\;) = \frac{1}{2} \mu(4 \;\cdot\;) + \frac{1}{2} \mu( 4 \;\cdot\; -2)$$
is a spectral measure with a spectrum
\begin{equation}\label{eq:Lambda-4-0-1}
  \Lambda = \bigcup_{n=1}^\f \big\{ \ell_1 + 4 \ell_2 + \cdots + 4^{n-1} \ell_n: \ell_1, \ell_2, \ldots, \ell_n \in \{0,1\} \big\},
\end{equation}
but the standard middle-third Cantor measure is not.
We refer readers to \cite{Hutchinson-1981} on self-similar sets and measures.
Note that the self-similar measure $\mu_{4,\{0,2\}}$ is mutually singular with respect to Lebesgue measure.
From then on, singularly continuous spectral measures have entered into the realm of fractal geometry and have been extensively explored \cite{An-Fu-Lai-2019,An-He-2014,An-He-He-2019,An-He-Lau-2015,Dutkay-Haussermann-Lai-2019,Laba-Wang-2002,Strichartz-2000, Fu-Wen-2017,Li-Miao-Wang-2021,Dai-Fu-Yan-2021,Dai-He-Lau-2014,Li-2009,Liu-Dong-Li-2017,Dutkay-Han-Sun-2009,Fu-He-2017, Fu-He-Wen-2018,He-Tang-Wu-2019,Li-Miao-Wang-2022,Dutkay-Jorgensen-2012,Li-Wu-2024,Liu-Lu-Zhou-2023b}.

Many surprising phenomena appear in singularly continuous spectral measures.
Besides the set $\Lambda$ defined in \eqref{eq:Lambda-4-0-1}, the sets $5\Lambda, 7\Lambda, 11\Lambda, 13\Lambda, 17\Lambda, \ldots$ are all spectra of $\mu_{4,\{0,2\}}$ \cite{Dutkay-Jorgensen-2012}. The scaling property was first found in \cite{Laba-Wang-2002} and is common to many singularly continuous spectral measures \cite{An-Dong-He-2022,He-Tang-Wu-2019,Fu-He-2017,Dutkay-Han-Sun-2009}.
When the spectrum exists, we could investigate the convergence of Fourier series of functions $$\sum_{\lambda \in \Lambda}~ \la f, e_\lambda \ra_{L^2(\mu)}\;e_\lambda(x).$$
However, the convergence of the mock Fourier series may be very different for distinct spectra of singularly continuous spectral measures \cite{Strichartz-2006,Dutkay-Han-Sun-2014,Fu-Tang-Wen-2022,Pan-Ai-2023}.
For the spectral measure $\mu_{4,\{0,2\}}$, Strichartz \cite{Strichartz-2006} proved that the mock Fourier series of continuous functions converges uniformly with respect to the spectrum $\Lambda$, but associated with the spectrum $17\Lambda$, Dutkey, Han and Sun \cite{Dutkay-Han-Sun-2014} showed that there exists a continuous function such that its mock Fourier series is divergent at $0$.
In addition, it was showed that for a class of Moran spectral measures, the Beurling dimension of spectra has the intermediate value property \cite{Li-Wu-2024}.
All of these findings suggest that spectra of singularly continuous spectral measures are more intricate than those of absolutely continuous spectral measures.
This motivates us to find or construct more singularly continuous spectral measures.

It has been proved that a compact supported spectral measure must be of pure type, that is, it is either discrete, or singularly continuous, or absolutely continuous \cite{He-Lai-Lau-2013}.
For absolutely continuous spectral measures, the density function must be constant on its support \cite{Dutkay-Lai-2014}, and thus this case is reduced to the original spectral set conjecture.
For discrete spectral measures, we introduce the concept of admissible pair, which is also used in convolution to construct singularly continuous spectral measures.

Our focus is on the real line.
For a finite subset $A \subset \R$, we define the discrete measure
$$\delta_A = \frac{1}{\# A} \sum_{a\in A} \delta_a, $$
where $\#$ denotes the cardinality of a set and $\delta_a$ denotes the Dirac measure concentrated on the point $a$.
Given an integer $N$ with $|N| \ge 2$ and a finite subset $B \subset\Z$ with $\# B \ge 2$, we say that $(N,B)$ is an \emph{admissible pair} if the discrete measure $\delta_{N^{-1} B}$ admits a spectrum $L\subset \Z$, that is, the matrix
\begin{equation}\label{unitary-matrix}
  \left( \frac{1}{\sqrt{\# B}} e^{-2\pi i \frac{b\ell}{N}} \right)_{b \in B, \ell \in L}
\end{equation}
is unitary. To emphasize the set $L$, $(N,B,L)$ is also called a \emph{Hadamard triple}.
If there are finitely many admissible pairs $\{(N_k,B_k)\}_{k=1}^n$, then the convolution
$$\mu_{n} = \delta_{N_1^{-1} B_1} * \delta_{(N_1 N_2)^{-1} B_2} * \cdots * \delta_{(N_1N_2 \cdots N_n)^{-1} B_n}$$
is a spectral measure with a spectrum
$$\Lambda_n = L_1 + N_1 L_2 + N_1 N_2 L_3 + \cdots + N_1 N_2 \cdots N_{n-1} L_n,$$
where $L_k \subset \Z$ is a spectrum of $\delta_{N_k^{-1} B_k}$ for all $1 \le k\le n$.
A natural question comes to mind.
\begin{quote}
  \textbf{Question}: \emph{Given a sequence of admissible pair $\{(N_k,B_k)\}_{k=1}^\f$, under what condition is the infinite convolution
  \begin{equation}\label{mu-infinite-convolution}
    \mu = \delta_{N_{1}^{-1} B_1} * \delta_{(N_1 N_2)^{-1} B_2} * \cdots * \delta_{(N_1 N_2 \cdots N_k)^{-1} B_k } * \cdots.
  \end{equation}
  a spectral measure ?}
\end{quote}

The infinite convolution generated by a sequence of admissible pairs was first raised by Strichartz \cite{Strichartz-2000} to construct more spectral measures.
If the infinite convolution defined in \eqref{mu-infinite-convolution} exists, then it must be of pure type \cite[Theorem~35]{Jessen-Wintner-1935}, and in most cases it is singularly continuous.
The admissible pair assumption implies the existence of an infinite mutually orthogonal set of exponential functions, but it is difficult to show the completeness.
When $(N_k,B_k) = (N,B)$ for all $k \ge 1$, the infinite convolution is reduced to self-similar measure
$$\mu_{N,B} = \delta_{N^{-1} B} * \delta_{N^{-2} B} * \cdots * \delta_{N^{-k}B} * \cdots.$$
{\L}aba and Wang \cite{Laba-Wang-2002} showed that if $(N,B)$ is an admissible pair, then the self-similar measure $\mu_{N,B}$ is a spectral measure, and Dutkay, Haussermann and Lai \cite{Dutkay-Haussermann-Lai-2019} generalized it to self-affine measures in higher dimensions.

However, the admissible pair assumption alone is not enough to guarantee that the corresponding infinite convolution is a spectral measure (see Example 4.3 in \cite{An-He-He-2019}), even if the sequence of admissible pairs is chosen from a finite set of admissible pairs (see Example 1.8 in \cite{Dutkay-Lai-2017}).
Nevertheless, it is widely believed that negative examples are very rare.
An, Fu and Lai \cite{An-Fu-Lai-2019} introduced the concept of equi-positivity, and used the integral periodic zero set to define an admissible family, both of which have been extensively employed in analysing the spectrality of infinite convolutions \cite{Li-Miao-Wang-2021,Li-Miao-Wang-2022,Lu-Dong-Zhang-2022,Liu-Lu-Zhou-2023b}.

In the paper, we first simply the admissible family condition for spectrality of infinite convolutions (see Theorem \ref{main-spectrality}).
Then we focus on the integral periodic zero set.
For general Borel probability measures, we use the structure of support to characterize the integral periodic zero set (see Theorem \ref{theorem-IPZS-empty-1}).
For infinite convolutions generated by admissible pairs, we generalize the argument for self-similar measures (see Theorem \ref{theorem-IPZS-empty-2} and the proof).
The analysis becomes more complicated due to the absence of self-similarity.
Next, the above results are applied to the random convolution generated by finitely many admissible pairs.
We show that almost all random convolutions are spectral measures (see Theorem \ref{general-theorem}), which primarily relies on the ergodic property of symbolic space and Theorem \ref{theorem-IPZS-empty-2}.
This means that the non-spectral random convolution is null in the sense of measure.
Finally, we study a special random convolution generated by two admissible pairs by analysing the support of measures (see Theorem \ref{special-case}).

\subsection{Main results}
We always assume that the infinite convolution $\mu$ defined in \eqref{mu-infinite-convolution} exists as a Borel probability measure in weak limit sense, see \cite{Li-Miao-Wang-2022} for the sufficient and necessary condition of weak convergence of infinite convolutions.
The infinite convolution $\mu$ may be rewritten as $\mu = \mu_{n} * \mu_{>n}$,
where $$\mu_n= \delta_{N_{1}^{-1} B_1} * \delta_{(N_1 N_2)^{-1} B_2} * \cdots * \delta_{(N_1 N_2 \cdots N_n)^{-1} B_n },$$
and $$ \mu_{>n} = \delta_{(N_1 N_2 \cdots N_{n+1})^{-1} B_{n+1} } * \delta_{(N_1 N_2 \cdots N_{n+2})^{-1} B_{n+2} } * \cdots. $$
Then for $n \ge 1$, we define
\begin{equation}\label{nu-large-than-n}
  \nu_{>n}(\;\cdot\;) = \mu_{>n}\left( \frac{1}{N_1 N_2 \cdots N_n} \; \cdot\; \right),
\end{equation}
that is, $$\nu_{>n}=\delta_{N_{n+1}^{-1} B_{n+1}} * \delta_{(N_{n+1} N_{n+2})^{-1} B_{n+2}} * \cdots.$$

We write $\mathcal{P}(\R)$ for the set of all Borel probability measures on $\R$.
For $\nu \in \mathcal{P}(\R)$, we write
\begin{equation}\label{integral-periodic-zero}
  \mcal{Z}(\nu) = \set{\xi \in \R:\; \wh{\nu}(\xi+k) = 0 \text{ for all } k \in \Z}
\end{equation}
for the {\it integral periodic zero set} of Fourier transform of $\nu$.
Using the integral periodic zero set, An, Fu and Lai \cite{An-Fu-Lai-2019} defined the admissible family to analysis the spectrality.
Here, we simply the admissible family condition.

\begin{theorem}\label{main-spectrality}
  Given a sequence of admissible pair $\{(N_k,B_k)\}_{k=1}^\f$, suppose that the infinite convolution $\mu$ defined by \eqref{mu-infinite-convolution} exists, and the sequence $\{ \nu_{>n} \}$ is defined by \eqref{nu-large-than-n}.
  If there exists a subsequence $\{\nu_{>n_j}\}$ which converges weakly to $\nu$, and $\mcal{Z}(\nu)=\emptyset$, then $\mu$ is a spectral measure with a spectrum in $\Z$.
\end{theorem}

The integral periodic zero set plays an crucial role in determining the spectrality of infinite convolutions.
It has been showed in \cite{An-Fu-Lai-2019} that if $\nu \in \mathcal{P}(\R)$ with $\mathrm{spt}(\nu) \subset [0,1]$, then $\mathcal{Z}(\nu) = \emptyset$ if and only if $\nu = \frac{1}{2} \delta_0 + \frac{1}{2} \delta_1$.
It was also pointed out that for the Borel probability measure outside $[0,1]$, the integral periodic zero set cannot be easily analyzed.
We use the structure of support to characterize the integral periodic zero set.

\begin{theorem}\label{theorem-IPZS-empty-1}
  Let $\nu \in \mathcal{P}(\R)$.
  If there exists a Borel subset $E \subset \R$ such that $\nu(E)>0$, and $$ \nu( E+k ) =0 $$ for all $k \in \Z\setminus\{0\}$,
  then we have that $\mcal{Z}(\nu) = \emptyset$.
\end{theorem}

For infinite convolutions generated by a sequence of admissible pairs, we generalize the argument for self-similar measures in \cite[Theorem~5.4]{Dutkay-Haussermann-Lai-2019}.
The analysis becomes more complicated due to the absence of self-similarity.
For a finite subset $B \subset \Z$, we set
$$ M_{B}(\xi) = \frac{1}{\# B} \sum_{b \in B} e^{-2\pi i b \xi}. $$
For a function $f:\R \to \C$, we denote the zero set of $f$ by
$$\mathcal{O}(f)=\big\{ x\in \R: f(x) =0 \big\}. $$
A set $A \subset \R$ is called \emph{discrete} if the set $A$ has no accumulation points, that is, for $h >0$ the set $[-h,h] \cap A$ is finite.

\begin{theorem}\label{theorem-IPZS-empty-2}
  Let $\{(N_k, B_k)\}_{k=1}^\f$ be a sequence of admissible pairs.
  The infinite convolution $\mu$ and the sequence $\{\nu_{>n}\}$ are defined by \eqref{mu-infinite-convolution} and \eqref{nu-large-than-n}.
  Suppose that

  {\rm(i)} there exists a weak convergent subsequence $\{\nu_{>n_j}\}$;

  {\rm(ii)} the set $$\bigcup_{k=1}^\f N_k \mcal{O}(M_{B_k})$$ is discrete;

  {\rm(iii)} for each $k \ge 1$, $$\mathrm{gcd}\left( \bigcup_{j=k}^\f (B_j - B_j) \right) =1.$$

  \noindent
  Then we have that $\mcal{Z}(\mu)=\emptyset$.
\end{theorem}

Note that the set $\mathcal{O}(M_B)$ is discrete for every finite subset $B \subset \Z$ because $M_B(\xi)$ is extendable to an entire function on the complex plane.
If the sequence of admissible pairs $\set{(N_k, B_k)}_{k=1}^\f$ is chosen from a finite set of admissible pairs, then it is clear that the conditions {\rm(i)} and {\rm(ii)} in Theorem \ref{theorem-IPZS-empty-2} hold.
We immediately have the following corollary.
\begin{corollary}\label{coro-finite-admissible-pair}
  Suppose that the sequence of admissible pairs $\{ (N_k, B_k)\}_{k=1}^\f$ is chosen from a finite set of admissible pairs.
  If for each $k \ge 1$, $$\mathrm{gcd}\left( \bigcup_{j=k}^\f (B_j - B_j) \right) =1,$$
  then we have that $\mcal{Z}(\mu)=\emptyset$.
\end{corollary}

Next we apply the above results to random convolutions.
Let $\{ (N_j, B_j) \}_{j=1}^m$ be finitely many admissible pairs.
Let $\Omega=\{1,2,\ldots, m\}^\N$ be the symbolic space over the alphabet $\{1,2,\ldots, m\}$.
For $\omega = (\omega_k)_{k=1}^\f \in \Omega$, we define the random convolution
$$\mu_\omega = \delta_{ N_{\omega_1}^{-1} B_{\omega_1} } * \delta_{(N_{\omega_1} N_{\omega_2})^{-1} B_{\omega_2} } * \cdots * \delta_{ (N_{\omega_1} N_{\omega_2} \cdots N_{\omega_k})^{-1} B_{\omega_k} } * \cdots. $$
More generally, given a sequence of positive integers $\{n_k\}$, for $\omega = (\omega_k)_{k=1}^\f \in \Omega$, we define the infinite convolution
$$ \mu_{\omega,\{n_k\}} = \delta_{ N_{\omega_1}^{-n_1} B_{\omega_1} } * \delta_{ N_{\omega_1}^{-n_1} N_{\omega_2}^{-n_2} B_{\omega_2} } * \cdots * \delta_{ N_{\omega_1}^{-n_1} N_{\omega_2}^{-n_2} \cdots N_{\omega_k}^{-n_k} B_{\omega_k} } * \cdots. $$
Note that $\{(N_{\omega_k}^{n_k}, B_{\omega_k})\}_{k=1}^\f$ is a sequence of admissible pairs.
If $n_k =1$ for all $k\ge 1$, then we have that $\mu_{\omega,\{n_k\}} = \mu_{\omega}$.
However, if the sequence $\{n_k\}$ is unbounded, the infinite convolution $\mu_{\omega,\{n_k\}}$ is not generated by finitely many admissible pairs.

The random convolution was first studied by Strichartz in \cite{Strichartz-2000}, where he constructed the spectrum under a specific uniform separation condition. But in general, this condition is challenging to verify.
If finitely many admissible pairs $\{(N,B_j)\}_{j=1}^m$ satisfy that the set $L \subset \Z$ is a common spectrum for all discrete measures $\delta_{N^{-1}B_j},1 \le j \le m$, then An, He and Lau \cite{An-He-Lau-2015} could construct the spectrum under the condition that $L+L \subset \{0,1,\ldots,N - 1\}$, and Dutkay and Lai \cite{Dutkay-Lai-2017} showed that almost all random convolutions admit a common spectrum.
The authors \cite{Li-Miao-Wang-2021} showed that if $\gcd(B_j - B_j) =1$ for all $1 \le j \le m$,
then all infinite convolutions $\mu_{\omega,\{n_k\}}$ are spectral measures.
Considering the Bernoulli measure on symbolic space, we prove that almost all random convolutions are spectral measures.

\begin{theorem}\label{general-theorem}
  Given finitely many admissible pairs $\set{(N_j, B_j)}_{j=1}^m$ and a sequence of positive integers $\{n_k\}$, for every Bernoulli measure $\PP$ on $\Omega$, the infinite convolution $\mu_{\omega,\{n_k\}}$ is a spectral measure for $\PP$-a.e. $\omega \in \Omega$.
  In particular, the random convolution $\mu_\omega$ is a spectral measure for $\PP$-a.e. $\omega \in \Omega$.
\end{theorem}

Generally, it is difficult to improve the ``almost all'' answer to a deterministic answer.
For the following special case, we give a complete characterization of spectrality of random convolutions.

\begin{theorem}\label{special-case}
  Given an integer $t\ge 1$ and two coprime integers $N,p \ge 2$, let $$\Omega = \{1,2\}^\N,\; N_1 = N_2 = tN,\; B_1=\set{0,1,\ldots, N - 1},\; B_2 =p \cdot \set{0,1,\ldots, N - 1} .$$

  {\rm(i)} If $t=1$, then for $\omega\in \Omega$, the infinite convolution $\mu_\omega$ is a spectral measure if and only if $\omega = 2^\f$ or the symbol ``1" occurs infinitely many times in $\omega$.

  {\rm(ii)} If $t \ge 2$, then for all $\omega\in \Omega$, the infinite convolution $\mu_\omega$ is a spectral measure.
\end{theorem}

The rest of paper is organized as follows.
In Section \ref{section-def-pre}, we introduce some definitions and some known results.
In Section \ref{section-infinite-convolution}, we give the proof of Theorem \ref{main-spectrality}.
In Section \ref{section-integral-periodic-zero}, we investigate the integral periodic zero set of Fourier transform and prove Theorem \ref{theorem-IPZS-empty-1} and \ref{theorem-IPZS-empty-2}.
The proofs of Theorem \ref{general-theorem} and \ref{special-case} are given in Section \ref{section-random-convolution} and \ref{section-special-case}, respectively.
Finally, we give some examples of spectral measures in Section \ref{section-example}.

\section{Definitions and preliminaries}\label{section-def-pre}

We first introduce the symbolic space.
The symbolic space $\Omega=\{1,2,\ldots,m\}^\N$, equipped with a metric $$\rho(\omega,\eta) = 2^{-\min\{k \ge 1: \omega_k \ne \eta_k\}}$$ for $\omega = (\omega_k)_{k=1}^\f,\eta = (\eta_k)_{k=1}^\f \in \Omega$,
is a compact metric space.
For a sequence $\set{ \omega(j) }$ in $\Omega$, we have that $\omega(j)$ converges to $\eta$ in $\Omega$ if and only if for $k\ge 1$ there exists $j_0 \ge 1$ such that for all $j \ge j_0$, $$\omega_1(j) \omega_2(j) \cdots \omega_k(j) = \eta_1 \eta_2 \cdots \eta_k.$$

For $i_1, i_2, \ldots, i_n \in \{1,2,\ldots,m\}$, we define the $n$-level \emph{cylinder} $$[i_1 i_2 \cdots i_n] = \set{(\omega_k)_{k=1}^\f \in \Omega: \omega_j = i_j \text{ for } j=1,2,\ldots, n}.$$
Given a probability vector $(p_1, p_2,\ldots, p_m)$, we define a probability measure $\PP$ on $\Omega$ by $$\PP([i_1 i_2 \cdots i_n]) = p_{i_1} p_{i_2} \cdots p_{i_n}$$
for all $i_1, i_2, \ldots, i_n \in \{1,2,\ldots,m\}$.
The probability measure $\PP$ is called the \emph{Bernoulli measure} associated with the probability vector $(p_1, p_2,\ldots, p_m)$.
A probability vector $(p_1, p_2,\ldots, p_m)$ is called \emph{positive} if $p_j >0$ for all $1\le j \le m$.

For $\mu \in \mcal{P}(\R)$, the \emph{Fourier transform} of $\mu$ is given by
$$ \wh{\mu}(\xi) = \int_{\R} e^{-2\pi i \xi x} \D \mu(x),\; \xi \in \R. $$
It is easy to verify that $\wh{\mu}(\xi)$ is uniformly continuous on $\R$ and $\wh{\mu}(0) =1$.

For $\mu\in \mathcal{P}(\R)$, the \emph{support} of $\mu$ is defined to be the smallest closed subset with full measure, i.e., $$\mathrm{spt}(\mu) = \R \setminus \bigcup \left\{ U \subset \R: U \text{ is open, and } \mu(U)=0 \right\}.$$

Let $\mu,\mu_1,\mu_2,\ldots \in \mcal{P}(\R)$. We say that $\mu_n$ \emph{converges weakly} to $\mu$ if $$\lim_{n \to \f} \int_{\R} f(x) \D \mu_n(x) = \int_{\R} f(x) \D \mu(x)$$
for all $ f \in C_b(\R),$ where $C_b(\R)$ is the set of all bounded continuous functions on $\R$.
The weak convergence can be characterized by Fourier transform, see section 1.6 in \cite{Bogachev-2018} for details.
\begin{theorem}\label{weak-convergence-theorem}
  Let $\mu,\mu_1,\mu_2,\ldots \in \mcal{P}(\R)$.
  Then $\mu_n$ converges weakly to $\mu$ if and only if $\displaystyle \lim_{n \to \f} \wh{\mu}_n(\xi)=\wh{\mu}(\xi)$ for every $\xi \in \R$.
  Moreover, if $\mu_n$ converges weakly to $\mu$, then for $h>0$ we have that $\wh{\mu}_n(\xi)$ converges to $ \wh{\mu}(\xi)$ uniformly for $\xi \in [-h,h]$.
\end{theorem}

For $\mu,\nu \in \mcal{P}(\R)$, the \emph{convolution} $\mu *\nu$ is given by $$ \mu*\nu(B) = \int_{\R} \nu(B-x) \D \mu(x)= \int_{\R} \mu(B-y) \D \nu(y), $$
for every Borel subset $B\subset \R$.
Equivalently, the convolution $\mu *\nu$ is the unique Borel probability measure satisfying $$\int_{\R} f(x) \D \mu*\nu(x) = \int_{\R \times \R} f(x+y) \D \mu \times \nu(x,y),$$
for  all $ f \in C_b(\R).$
It is easy to check that $$\wh{\mu *\nu}(\xi) = \wh{\mu}(\xi) \wh{\nu}(\xi).$$
The following criterion is frequently employed to verify the spectrality of measures \cite{Jorgensen-Pedersen-1998,Li-Miao-Wang-2021}.
\begin{theorem} \label{spectrality-Q}
  Let $\mu \in \mathcal{P}(\R)$ and let $\Lambda \subset \R$ be a countable subset.
  Define $$Q(\xi) = \sum_{\lambda\in \Lambda} |\wh{\mu}(\xi + \lambda)|^2.$$
  Then the set $\Lambda$ is a spectrum of $\mu$ if and only if $Q(\xi)=1$ for all $\xi \in \R$.
\end{theorem}

The spectrality of measures is invariant under linear transformations.
For $a,b \in \R$ with $a\ne 0$, we define $T_{a,b}:\R \to \R$ by
\begin{equation}\label{linear-transformation}
  T_{a,b}(x) = ax+b.
\end{equation}

\begin{lemma}\label{spectrality-invariant}
  If $\mu \in \mathcal{P}(\R)$ is a spectral measure with a spectrum $\Lambda$,
  then the measure $\mu \circ T_{a,b}^{-1}$ is a spectral measure with a spectrum $\Lambda/a$ for $a,b \in \R$ with $a\ne 0$.
\end{lemma}
\begin{proof}
  Write $\nu = \mu \circ T_{a,b}^{-1}$ and $\Lambda' = \Lambda/a$.
  Obviously, we have that $$\wh{\nu}(\xi) = e^{-2\pi i b \xi}\wh{\mu}(a\xi).$$
  Since $\Lambda$ is a spectrum of $\mu$, by Theorem \ref{spectrality-Q}, we have that
  \begin{align*}
    Q(\xi) & = \sum_{\lambda \in \Lambda'} |\wh{\nu}(\xi+\lambda)|^2 \\
     & = \sum_{\lambda \in \Lambda'} |e^{-2\pi i b (\xi+\lambda)}\wh{\mu}(a\xi+a\lambda)|^2 \\
     & = \sum_{\lambda \in \Lambda'} |\wh{\mu}(a\xi+a\lambda)|^2 \\
     & = \sum_{\lambda \in \Lambda} |\wh{\mu}(a\xi+\lambda)|^2 \\
     & = 1.
  \end{align*}
  It follows from Theorem \ref{spectrality-Q} that $\Lambda'$ is a spectrum of $\nu$.
\end{proof}

In the end, we list some simple properties of admissible pairs, which are needed in our proofs.
\begin{lemma}\label{admissible-pair-lemma}
  Suppose that $(N,B)$ is an admissible pair, and $L\subset \Z$ is a spectrum of the discrete measure $\delta_{N^{-1} B}$.

  {\rm(i)} The elements in $L$ are distinct module $N$.

  {\rm(ii)} If $\widetilde{L} \equiv L \pmod{N}$, then $\widetilde{L}$ is also a spectrum of $\delta_{N^{-1} B}$.

  {\rm(iii)} For $b \in \Z$, $(N, B-b)$ is an admissible pair.

  {\rm(iv)} If $d \mid \mathrm{gcd}(B)$, then $( N, \frac{1}{d} B)$ is an admissible pair.

  {\rm(v)} For $t\in \Z$ with $t\ne 0$, $(tN, B)$ is an admissible pair
\end{lemma}
\begin{proof}
  (i) It follows from the fact that the matrix in \eqref{unitary-matrix} is unitary.

  (ii) Since $\widetilde{L} \equiv L \pmod{N}$, we have that the matrix
  $$ \left( \frac{1}{\sqrt{\# B}} e^{-2\pi i \frac{b\tilde{\ell}}{N}} \right)_{b \in B, \tilde{\ell} \in \widetilde{L}} = \left( \frac{1}{\sqrt{\# B}} e^{-2\pi i \frac{b\ell}{N}} \right)_{b \in B, \ell \in L} $$
  is unitary. It follows that $\widetilde{L}$ is also a spectrum of $\delta_{N^{-1} B}$.

  Note that $$ \delta_{N^{-1}(B-b)} = \delta_{N^{-1} B} \circ T_{1,-b/N}^{-1},\; \delta_{N^{-1}(\frac{1}{d} B)} = \delta_{N^{-1} B} \circ T_{1/d,0}^{-1},\; \delta_{(tN)^{-1} B} = \delta_{N^{-1} B} \circ T_{1/t,0}^{-1}.$$
  (iii), (iv) and (v) follow immediately from Lemma \ref{spectrality-invariant}.
\end{proof}

\section{Spectrality of infinite convolutions}\label{section-infinite-convolution}

The equi-positivity was first raised by An, Fu and Lai \cite{An-Fu-Lai-2019} to study the spectrality of infinite convolutions, and the authors \cite{Li-Miao-Wang-2021} generalized it into the following form.
\begin{definition}\label{def-equipositive}
  We call $\Phi \subset \mcal{P}(\R)$ an equi-positive family if there exists $\ep>0$ and $\gamma>0$ such that for $x\in [0,1)$ and $\mu\in \Phi$ there exists an integer $k_{x,\mu} \in \Z$ such that
  $$ |\wh{\mu}(x+y+k_{x,\mu})| \ge \ep,$$
  for all $ |y| <\gamma$, where $k_{x,\mu} =0$ for $x=0$.
\end{definition}

The following theorem is very useful to study the spectrality of infinite convolutions, see \cite{Li-Miao-Wang-2021} for the proof. It was proved in \cite{An-Fu-Lai-2019} under the no-overlap condition.

\begin{theorem}\label{spectrality-1}
  Given a sequence of admissible pairs $\{(N_k,B_k)\}_{k=1}^\f$, suppose that the infinite convolution $\mu$ defined by \eqref{mu-infinite-convolution} exists, and the sequence $\{\nu_{>n}\}$ is defined by \eqref{nu-large-than-n}.
  If there exists a subsequence $\{\nu_{>n_j}\}$ which is an equi-positive family, then $\mu$ is a spectral measure with a spectrum in $\Z$.
\end{theorem}

\begin{proof}[Proof of Theorem \ref{main-spectrality}]
  By Theorem \ref{spectrality-1}, it suffices to show that there exists $j_0 \ge 1$ such that the family $\{\nu_{> n_j}\}_{j=j_0}^\f$ is equi-positive.

  Since $\mathcal{Z}(\nu)=\emptyset$, for each $x\in [0,1]$, there exists $k_x\in \Z$ such that $\wh{\nu}(x+k_x) \ne 0$.
  Since $\widehat{\nu}(\xi)$ is continuous, there exists $\ep_x>0$ and $\gamma_x >0$ such that
  \begin{equation}\label{positive-1}
    |\wh{\nu}(x+k_x+y)|\ge \ep_x
  \end{equation}
  for all $|y|<\gamma_x$.
  Note that $$[0,1] \subset \bigcup_{x\in[0,1]} (x-\gamma_x/2, x+\gamma_x/2). $$
  By the compactness of $[0,1]$, there exist finitely many $x_1, x_2, \ldots, x_q \in [0,1]$ such that
  \begin{equation}\label{finite-cover}
    [0,1] \subset \bigcup_{\ell=1}^q (x_\ell-\gamma_{x_\ell}/2, x_\ell +\gamma_{x_\ell}/2).
  \end{equation}
  Since $\wh{\nu}(0) =1$ and $\wh{\nu}(\xi)$ is continuous, there exists $\gamma_0>0$ such that
  \begin{equation}\label{positive-2}
    |\wh{\nu}(y)| \ge 1/2
  \end{equation}
  for all $|y| <\gamma_0$.

  Let $\ep = \min\set{ 1/4, \ep_{x_1}/2, \ep_{x_2}/2, \ldots, \ep_{x_q}/2 }$ and $\gamma= \min \set{ \gamma_0, \gamma_{x_1}/2, \gamma_{x_2}/2, \cdots, \gamma_{x_q}/2 }$.
  Let $h = 1 + \gamma+\max\set{|k_{x_1}|, |k_{x_2}|, \ldots, |k_{x_q}|}$.
  Since $\{ \nu_{>n_j} \}$ converges weakly to $\nu$, by Theorem \ref{weak-convergence-theorem}, we have that $\wh{\nu}_{>n_j}(\xi)$ converges to $\wh{\nu}(\xi)$ uniformly on $[-h, h]$.
  Thus, there exists $j_0 \ge 1$ such that
  \begin{equation}\label{converge-uniformly}
    |\wh{\nu}_{>n_j}(\xi) - \wh{\nu}(\xi)| < \ep
  \end{equation}
  for all $j \ge j_0$ and all $\xi \in [-h,h]$.

  For each $x\in (0,1)$, by \eqref{finite-cover}, we may find $1\le \ell \le q$ such that $|x-x_\ell|<\gamma_{x_\ell}/2$.
  For $j \ge j_0$ and $|y|<\gamma$, noting that $|x+k_{x_\ell} +y|<h$, it follows from \eqref{converge-uniformly} that $$|\wh{\nu}_{>n_j}(x+k_{x_\ell} + y)| \ge |\wh{\nu}(x+k_{x_\ell} +y)| -\ep. $$
  Since $|x-x_\ell +y| < \gamma_{x_\ell}/2 +\gamma \le \gamma_{x_\ell} $, by \eqref{positive-1}, we have that $$|\wh{\nu}(x+k_{x_\ell} +y)|= |\wh{\nu}(x_\ell+k_{x_\ell} +x-x_\ell+y)|\ge \ep_{x_\ell} \ge 2\ep.$$
  Thus, for $j \ge j_0$ and $|y|<\gamma$, $$|\wh{\nu}_{>n_j}(x+k_{x_\ell} + y)| \ge \ep. $$
  For $x=0$, it follows from \eqref{converge-uniformly} and \eqref{positive-2} that for $j \ge j_0$ and for $|y| <\gamma$, $$|\wh{\nu}_{>n_j}(y)| \ge |\wh{\nu}(y)|-\ep \ge 1/4 \ge \ep.$$
  Therefore, we conclude that the family $\{ \nu_{> n_j} \}_{j=j_0}^\f$ is equi-positive.
\end{proof}

\section{The integral periodic zero set}\label{section-integral-periodic-zero}

In this section, we study the integral periodic zero set of Fourier transform.
Let $\mathbb{T} = \R/\Z$, and write $\mathcal{M}(\mathbb T)$ for the set of all complex Borel measures on $\mathbb{T}$. The following is  the uniqueness theorem of Fourier coefficients in classical harmonic analysis.

\begin{theorem}\cite[Corollary~7.1]{Katznelson-2004}\label{uniqueness-theorem}
  Let $\nu \in \mathcal{M}(\mathbb T)$.
  If the Fourier coefficients $$\widehat{\nu}(k) = \int_{\mathbb{T}} e^{-2\pi i k x} \D\nu(x) =0$$ for all $k \in \Z$, then $\nu = 0$.
\end{theorem}

\begin{proof}[Proof of Theorem \ref{theorem-IPZS-empty-1}]
  Since $\nu(E) >0$, there exists $k_0 \in \Z$ such that $\nu\big( E \cap [k_0,k_0 +1) \big)>0$.
  Let $\widetilde{E}=  (E-k_0)\cap [0,1)$ and $\widetilde{\nu} = \nu*\delta_{\{-k_0\}}$.
  For any Borel subset $F \subset \R$, we have that $$\widetilde{\nu}(F) = \nu*\delta_{\{-k_0\}}(F) = \nu(F+k_0). $$
  It follows that $\widetilde{\nu}\big( \widetilde{E} \big) = \nu\big( E \cap [k_0,k_0 +1) \big)>0$, and
  $$\widetilde{\nu}\big( \widetilde{E} +k \big) = \nu\big((E +k)\cap [k+k_0,k+k_0+1)\big)=0$$ for all $k \in \Z \setminus \{0\}$.
  Noting that $\widehat{\widetilde{\nu}}(\xi) = e^{2 \pi i k_0 \xi} \widehat{\nu}(\xi)$, we have that $\mathcal{Z} ( \widetilde{\nu} ) = \mathcal{Z}(\nu)$.
  Since $\widetilde{E} \subset [0,1)$, we can assume that $E \subset [0,1)$.
  Recall that a Borel probability measure on $\R$ is regular, see \cite[Theorem 2.18]{Rudin-1987}.
  We have that $\nu(E) = \sup\set{\nu(K): K \subset E \text{ is compact}}$.
  Therefore, in the following, we assume that $E \subset [0,1)$ is compact.

  For $\xi \in \R$, we define a complex measure $\nu_\xi$ on $\R$ by $$\frac{\D \nu_\xi}{\D \nu} = e^{-2\pi i\xi x}. $$
  Consider the natural homomorphism $\pi: \R \to \mathbb{T}$, and let $\rho_\xi = \nu_\xi \circ \pi^{-1}$ be the image measure on $\mathbb{T}$ of $\nu_\xi$ by $\pi$, i.e., for each Borel subset $F \subset \mathbb{T}$, $$\rho_\xi(F) = \nu_\xi(F+\Z) = \sum_{k\in \Z}\nu_\xi(F+k).$$

  Assume that $\xi \in \mathcal{Z}(\nu)$.
  For each $k \in \Z$, we have that
  \begin{align*}
    \widehat{\rho}_\xi(k)
    &=\int_{\mathbb{T}} e^{-2\pi i k x} \D \nu_\xi\circ \pi^{-1}(x)  \\
    & = \int_{\R} e^{-2\pi i k \pi(x)} \D \nu_\xi(x) \\
    & = \int_{\R} e^{-2\pi i k x} \D \nu_\xi(x) \\
    & = \int_{\R} e^{-2\pi i k x}\cdot e^{-2\pi i \xi x} \D \nu(x) \\
    & = \widehat{\nu}(\xi+k) =0,
  \end{align*}
  where the last equality follows from $\xi \in \mathcal{Z}(\nu)$.
  By Theorem \ref{uniqueness-theorem}, we conclude that $\rho_\xi =0$.
  It follows that
  \begin{align*}
    0 = \rho_\xi(E) = \nu_\xi(E+\Z) = \int_{E + \Z} e^{-2\pi i \xi x} \D\nu(x).
  \end{align*}
  Since $\nu\big( E + k \big) =0$ for all $k \in \Z\setminus\{0\}$, we obtain that $$ \int_E e^{-2\pi i \xi x} \D\nu(x) =0.$$
  Let $\tau$ be the normalized measure of $\nu$ on $E$, i.e., $\tau(\;\cdot\;) = \frac{1}{\nu(E)} \nu(\;\cdot\; \cap E)$.
  Then we have that $$\widehat{\tau}(\xi) = \frac{1}{\nu(E)} \int_E e^{-2\pi i \xi x} \D\nu(x) =0.$$
  Therefore, we obtain that $\widehat{\tau}(\xi)=0$ for all $\xi \in \mathcal{Z}(\nu)$.

  Suppose that $\mathcal{Z}(\nu) \ne \emptyset$ and take $\xi \in \mathcal{Z}(\nu)$.
  Since $\xi+k \in \mathcal{Z}(\nu)$ for all $k \in \Z$, we have that for all $k \in \Z$, $$\widehat{\tau}(\xi+k) =0.$$
  Consider the complex measure $\tau_\xi$ defined by
  \begin{equation}\label{nu-xi}
    \frac{\D \tau_\xi}{\D \tau} = e^{-2\pi i \xi x}.
  \end{equation}
  Since $\mathrm{spt}(\tau) \subset [0,1)$, $\tau_\xi$ can be viewed as a complex measure on $\mathbb{T}$.
  Moreover, the Fourier coefficients
  $$ \widehat{\tau}_\xi(k) = \widehat{\tau}(\xi+k)=0 $$ for all $k \in \Z$.
  By Theorem \ref{uniqueness-theorem}, we have that $\tau_\xi=0$.
  But, by \eqref{nu-xi} and Theorem 6.13 in \cite{Rudin-1987}, we have that the total variation $|\tau_\xi| = \tau \ne 0$, which leads to a contradiction.
  Therefore, we conclude that $\mathcal{Z}(\nu) =\emptyset$.
\end{proof}

Next, we focus on the infinite convolution generated by admissible pairs.
For convenience, in the rest of this section, we always assume that $\{ (N_k, B_k) \}_{k=1}^\f$ is a sequence of admissible pairs,
that the infinite convolution $$\mu = \delta_{N_{1}^{-1} B_1} * \delta_{(N_1 N_2)^{-1} B_2} * \cdots * \delta_{(N_1 N_2 \cdots N_n)^{-1} B_n } * \cdots$$ exists,
and that for every $n \ge 1$, $$\nu_{>n}=\delta_{N_{n+1}^{-1} B_{n+1}} * \delta_{(N_{n+1} N_{n+2})^{-1} B_{n+2}} * \cdots.$$
In order to prove Theorem \ref{theorem-IPZS-empty-2}, we need the following two lemmas.

\begin{lemma}\label{lemma-zero-set}
  Suppose that the condition {\rm(i)} in Theorem \ref{theorem-IPZS-empty-2} holds.
  Then for each $n\ge 1$ we have that $$\mathcal{O}(\wh{\nu}_{>n}) = \bigcup_{k=1}^\f N_{n+1} N_{n+2} \cdots N_{n+k} \mathcal{O}(M_{B_{n+k}}).$$
\end{lemma}
\begin{proof}
  Fix $n \ge 1$.
  Note that for $q \ge 1$, we have that $$\wh{\nu}_{>n}(\xi) = \prod_{k=1}^{q} M_{B_{n+k}} \left( \frac{\xi}{N_{n+1} N_{n+2} \cdots N_{n+k}} \right) \cdot \wh{\nu}_{> n+q} \left( \frac{\xi}{N_{n+1} N_{n+2} \cdots N_{n+q}}  \right). $$
  It follows that $$ \bigcup_{k=1}^\f N_{n+1} N_{n+2} \cdots N_{n+k} \mathcal{O}(M_{B_{n+k}}) \subset \mathcal{O}(\wh{\nu}_{>n}). $$

  Let $\nu$ denote the weak limit of $\{\nu_{>n_j} \}$.
  Since $\widehat{\nu}(0) =1$ and $\widehat{\nu}(\xi)$ is continuous, there exists $h > 0$ such that $|\wh{\nu}(\xi)| \ge 1/2$ for all $|\xi|\le h$.
  Since $\{ \nu_{>n_j} \}$ converges weakly to $\nu$, by Theorem \ref{weak-convergence-theorem}, we have that $ \wh{\nu}_{>n_j}(\xi)$ converges to $\wh{\nu}(\xi)$ uniformly on $[-h, h]$.
  Hence, there exists $j_0 \ge 1$ such that for all $j \ge j_0$ and all $|\xi|\le h$,
  $$ |\wh{\nu}_{>n_j}(\xi) - \wh{\nu}(\xi)| \le 1/4. $$
  Thus, we obtain that for $j \ge j_0$ and $|\xi|\le h$, $$|\wh{\nu}_{>n_j}(\xi)| \ge 1/4.$$

  Assume that $\wh{\nu}_{>n}(\xi_0) =0$.
  Choose a sufficiently large  integer $j \ge j_0$ such that $n_{j} >n$ and $$ \left|\frac{\xi_0}{N_{n+1} N_{n+2} \cdots N_{n_{j}} } \right| < h.$$
  It follows that $$  \left| \wh{\nu}_{>n_{j}} \left( \frac{\xi_0}{N_{n+1} N_{n+2} \cdots N_{n_{j}} } \right) \right| \ge \frac{1}{4}.$$
  Note that $$0=\wh{\nu}_{>n}(\xi_0) = \prod_{k=1}^{n_{j} -n} M_{B_{n+k}} \left( \frac{\xi_0}{N_{n+1} N_{n+2} \cdots N_{n+k}} \right) \cdot \wh{\nu}_{>n_{j}} \left( \frac{\xi_0}{N_{n+1} N_{n+2} \cdots N_{n_{j}}}  \right). $$
  Thus, we conclude that $$\xi_0 \in \bigcup_{k=1}^\f N_{n+1} N_{n+2} \cdots N_{n+k} \mathcal{O}(M_{B_{n+k}}).$$
  So we have that $$ \mathcal{O}(\wh{\nu}_{>n}) \subset \bigcup_{k=1}^\f N_{n+1} N_{n+2} \cdots N_{n+k} \mathcal{O}(M_{B_{n+k}}). $$
  This completes the proof.
\end{proof}

\begin{lemma}\label{lemma-finiteness}
  Suppose that the conditions {\rm(i)} and {\rm(ii)} in Theorem \ref{theorem-IPZS-empty-2} hold.
  Then for each $h >0$, there exists a constant $C>0$, depending only on $h$, such that for all $n \ge 1$, $$\#\big( [-h,h] \cap \mathcal{Z}(\nu_{>n}) \big) \le C.$$
\end{lemma}
\begin{proof}
  Noting that $\mathcal{Z}(\nu_{>n}) \subset \mathcal{O}(\wh{\nu}_{>n})$, it suffices to show that for all $n \ge 1$, $$\#\big( [-h,h] \cap \mathcal{O}(\wh{\nu}_{>n}) \big) \le C.$$

  Let $$D = \bigcup_{k=1}^\f N_k \mcal{O}(M_{B_k}). $$
  Since $0 \not\in \mcal{O}(M_{B_k})$ for each $k \ge 1$, we have that $0\not\in D$.
  Since $D$ is discrete, there exists $\delta>0$ such that $[-\delta,\delta] \cap D =\emptyset$.

  For $h >0$, we choose $k_0 \ge 1$ such that $2^{k_0} \delta > h$.
  By Lemma \ref{lemma-zero-set}, we have that $$ \mathcal{O}(\wh{\nu}_{>n}) \subset \bigcup_{k=1}^\f N_{n+1} N_{n+2} \cdots N_{n+k-1} D, $$ where $N_{n+1} N_{n+2} \cdots N_{n+k-1}=1$ for $k=1$.
  Noting that all $|N_k| \ge 2$, we have that for $k \ge k_0 +1$,
  $$\big( [-h,h] \cap N_{n+1} N_{n+2} \cdots N_{n+k-1} D \big) \ \subset\  N_{n+1} N_{n+2} \cdots N_{n+k-1} \big( [-\delta,\delta] \cap D \big)  = \emptyset.$$
  Thus, $$ [-h,h] \cap \mathcal{O}(\wh{\nu}_{>n}) \subset \bigcup_{k=1}^{k_0}\big( [-h,h] \cap N_{n+1} N_{n+2} \cdots N_{n+k-1} D \big).$$
  It follows that
  \begin{align*}
    \#\big( [-h,h] \cap \mathcal{O}(\wh{\nu}_{>n}) \big)
    & \le \sum_{k=1}^{k_0} \# \big( [-h,h] \cap N_{n+1} N_{n+2} \cdots N_{n+k-1} D \big)\\
    & \le k_0 \cdot \#\big( [-h,h]\cap D \big),
  \end{align*}
  where the constant $k_0 \cdot \#\big( [-h,h]\cap D \big)$ depends only on $h$.
\end{proof}

Now we are ready to prove Theorem \ref{theorem-IPZS-empty-2}.

\begin{proof}[Proof of Theorem \ref{theorem-IPZS-empty-2}]
   For $k\ge 1$, since $(N_k,B_k)$ is an admissible pair, by Lemma \ref{admissible-pair-lemma} (ii), we may find $L_k \subset \set{0,1,\ldots, |N_k|-1}$ such that the set $L_k$ is a spectrum of $\delta_{N_k^{-1} B_k}$.
   It follows from Theorem \ref{spectrality-Q} that for all $\xi\in \R$, $$ \sum_{\ell \in L_k} \left| M_{B_k}\left( \frac{\xi+\ell}{N_k} \right) \right|^2 =1.  $$
   For $\ell \in L_k$, we define $$ \tau_{\ell,k}(x) = N_k^{-1}(x+\ell). $$

  We prove this theorem by contradiction. Suppose that $\mcal{Z}(\mu) \ne \emptyset$.
  Arbitrarily choose $\xi_0 \in \mcal{Z}(\mu)$ and set $Y_0 = \set{\xi_0}$. For $n \ge 1$, we define
  $$ Y_n = \set{ \tau_{\ell,n}(\xi) : \;\xi \in Y_{n-1}, \; \ell \in L_{n}, \; M_{B_{n}}\big( \tau_{\ell,n}(\xi) \big) \ne 0 }.$$

  First, we show that for each $n \ge 1$, $$ \# Y_{n-1} \le \# Y_n. $$
  Since for each $ \xi \in Y_{n-1}$,
  $$ \sum_{\ell \in L_{n} } \left| M_{B_{n}}\big( \tau_{\ell,n}(\xi) \big) \right|^2 =1,$$
  there exists at least one element $\ell \in L_{n}$ such that $M_{B_{n}}\big( \tau_{\ell,n}(\xi) \big) \ne 0$.
  On the other hand, for $\ell_1 \ell_2 \cdots \ell_n \ne \ell_1' \ell_2' \cdots \ell_n'$ where $\ell_j, \ell_j' \in L_{j}$ for $1 \le j \le n$,
  we have that
  $$\tau_{\ell_n,n} \circ \cdots \circ \tau_{\ell_2,2} \circ \tau_{\ell_1,1} (\xi_0) \ne \tau_{\ell_n',n} \circ \cdots \circ \tau_{\ell_2',2} \circ \tau_{\ell_1',1}(\xi_0).$$
  Otherwise, $$ \frac{ \xi_0 + \ell_1 + N_{1} \ell_2 + \cdots + N_{1} N_{2} \cdots N_{n-1} \ell_n }{ N_{1} N_{2} \cdots N_{n} } = \frac{ \xi_0 + \ell_1' + N_{1} \ell_2' + \cdots + N_{1} N_{2}\cdots N_{n-1} \ell_n' }{N_{1} N_{2} \cdots N_{n}}, $$
  that is, $$  \ell_1 + N_{1} \ell_2 + \cdots + N_{1} N_{2} \cdots N_{n-1} \ell_n = \ell_1' + N_{1} \ell_2' + \cdots + N_{1} N_{2} \cdots N_{n-1} \ell_n'.$$
  Let $j_0 = \min\set{1 \le j \le n: \ell_j \ne \ell_j'}$.
  Then we have that $N_{j_0} \mid \ell_{j_0} - \ell_{j_0}'$.
  This contradicts with Lemma \ref{admissible-pair-lemma} (i).
  Therefore, we conclude that $\# Y_{n-1} \le \# Y_n$ for $n \ge 1$.

  Next, we prove that for each $n \ge 1$, $$ Y_n \subset \mcal{Z}(\nu_{>n}) .$$
  Write $\nu_{>0} = \mu$, and it is clear that  $Y_0 \subset \mcal{Z}(\nu_{>0})$.
  For $n \ge 1$, we assume that $Y_{n-1} \subset \mcal{Z}(\nu_{> n-1})$.
  For each $\tau_{\ell,n}(\xi) \in Y_{n}$ where $\xi \in Y_{n-1}, \; \ell \in L_{n}$ and $M_{B_{n}}\big( \tau_{\ell,n}(\xi) \big) \ne 0$, we have that for all $k \in \Z$,
  \begin{align*}
    0 & = \wh{\nu}_{> n-1}(\xi + \ell + N_{n} k) \\
    &= M_{B_{n}}\left( \frac{\xi + \ell}{ N_{n} } +k \right) \wh{\nu}_{>n} \left( \frac{\xi + \ell}{N_{n}} +k \right)\\
     & = M_{B_{n}}\big( \tau_{\ell,n}(\xi) \big) \wh{\nu}_{>n} \big( \tau_{\ell,n}(\xi) +k \big),
  \end{align*}
  where the last equality follows from integral periodicity of $M_{B_{n}}$.
  Since $M_{B_{n}}\big( \tau_{\ell,n}(\xi) \big) \ne 0$, we have that
  $$ \wh{\nu}_{>n} \big( \tau_{\ell,n}(\xi) +k \big) =0 ,$$ for all $k \in \Z$.
  This implies that $\tau_{\ell,n}(\xi) \in \mcal{Z}(\nu_{> n})$. Thus, $Y_{n} \subset \mcal{Z}(\nu_{> n})$.
  By induction, we obtain that $Y_n \subset \mcal{Z}(\nu_{> n})$ for all $n \ge 1$.

  For every $\xi \in Y_n$, by the definition of $Y_n$, we write $\xi$ as
  $$\xi = \tau_{\ell_n,n} \circ \cdots \circ \tau_{\ell_2,2} \circ \tau_{\ell_1,1} (\xi_0)
     = \frac{ \xi_0 + \ell_1 + N_{1} \ell_2 + \cdots + N_{1} N_{2}\cdots N_{n-1} \ell_n }{ N_{1} N_{2} \cdots N_{n} } .$$
  Since $|N_{j}|\geq 2$ and $0\le \ell_j < |N_{j}|$ for each $j \ge 1$, we have that
  $$ |\xi| \le \frac{|\xi_0|}{2^n} + \frac{1}{2^{n-1}} + \frac{1}{2^{n-2}} + \cdots + 1 \le |\xi_0| + 2. $$
  Let $h= |\xi_0| + 2$. Then, we have that $Y_n \subset [-h , h]$ for all $n \ge 1$.
  It follows that $$Y_n \subset [-h,h] \cap \mcal{Z}(\nu_{>n}).$$
  By the increasing cardinality of $Y_n$ and Lemma \ref{lemma-finiteness}, there exists $n_0 \ge 1$ such that $\# Y_n = \# Y_{n+1}$ for all $n\ge n_0$.

  Choose $\eta_0 \in Y_{n_0}$.
  Since $\# Y_{n_0+1} = \# Y_{n_0} $, there exists a unique $\ell_1 \in L_{n_0+1}$ such that  $M_{B_{n_0 +1}}\big( \tau_{\ell_1,n_0+1}(\eta_0) \big) \ne 0$.
  Note that $$\sum_{\ell \in L_{n_0+1}} \left| M_{B_{n_0+1}}\big( \tau_{\ell,n_0+1}(\eta_0) \big) \right|^2 =1. $$
  Let $\eta_1 = \tau_{\ell_1,n_0+1}(\eta_0)$. Then we have that $\eta_1 \in Y_{n_0+1}$ and $|M_{B_{n_0+1}}(\eta_1)| =1$, that is, $$ \left| \frac{1}{\# B_{n_0+1}} \sum_{b\in B_{n_0+1}} e^{-2\pi i b \eta_1 } \right|= \left| \frac{1}{\# B_{n_0+1}} \sum_{b\in B_{n_0+1}} e^{-2\pi i (b-b_1) \eta_1 } \right| =1, $$
  where $b_1 \in B_{n_0+1}$ is arbitrarily chosen.
  It follows that $(b-b_1)\eta_1 \in \Z$ for all $b \in B_{n_0+1}$.
  Thus, we have that $$\eta_1 \in \frac{1}{\gcd(B_{n_0+1} - B_{n_0+1})} \Z.$$
  Recursively, we define a sequence $\set{\eta_k}_{k=1}^\f$ such that $\eta_k = \tau_{\ell_k,n_0+k}(\eta_{k-1}) \in Y_{n_0+k}$ for some $\ell_k \in L_{n_0+k}$, and
  \begin{equation}\label{eta-k}
    \eta_k \in \frac{1}{\gcd(B_{n_0+k} - B_{n_0 +k}) } \Z.
  \end{equation}

  Since $\eta_1$ is a rational number, $\eta_0$ must be rational.
  Write $\eta_0 = q/p$ with $p\in \N$, $q\in \Z$, and $\gcd(p,q)=1$.
  Then for every $k \ge 1$, we have that
  \begin{align*}
    \eta_k & = \tau_{\ell_k, n_0+k} \circ \cdots \circ \tau_{\ell_2, n_0+2} \circ \tau_{\ell_1, n_0+1}(\eta_0) \\
    & = \frac{ \eta_0 + \ell_1 + N_{n_0+1} \ell_2 + \cdots + N_{n_0+1} N_{n_0+2} \cdots N_{n_0+k-1} \ell_k }{ N_{n_0+1} N_{n_0+2} \cdots N_{n_0+k} } \\
    & = \frac{ q + p(\ell_1 + N_{n_0+1} \ell_2 + \cdots + N_{n_0+1} N_{n_0+2} \cdots N_{n_0+k-1} \ell_k) }{ p N_{n_0+1} N_{n_0+2} \cdots N_{n_0+k} }.
  \end{align*}
  Note that $q + p(\ell_1 + N_{n_0+1} \ell_2 + \cdots + N_{n_0+1} N_{n_0+2} \cdots N_{n_0+k-1} \ell_k)$ and $p$ are coprime.
  It follows from \eqref{eta-k} that $$ p \mid \gcd(B_{n_0+k} - B_{n_0 +k}).  $$
  Thus, $$p \mid \mathrm{gcd}\left( \bigcup_{j=n_0+1}^\f (B_j - B_j) \right).$$
  By the condition (iii), we have that $p=1$.
  It follows that $\eta_0$ is an integer. This contradicts the facts that $\eta_0 \in \mcal{Z}(\nu_{> n_0})$ and $\mcal{Z}(\nu_{> n_0}) \cap \Z = \emptyset$.

  Therefore we obtain that $\mathcal{Z}(\mu) = \emptyset$. The proof is completed.
\end{proof}

\section{Spectrality of general random convolutions}\label{section-random-convolution}

Recall that $\{(N_j,B_j)\}_{j=1}^m$ are finitely many admissible pairs, and $\Omega = \{1,2,\ldots, m\}^\N$.
Given a sequence of positive integers $\{n_k\}$, for $\omega = (\omega_k)_{k=1}^\f \in \Omega$, we define the infinite convolution
\begin{equation}\label{mu-omega-n-k}
  \mu_{\omega,\{n_k\}} = \delta_{ N_{\omega_1}^{-n_1} B_{\omega_1} } * \delta_{ N_{\omega_1}^{-n_1} N_{\omega_2}^{-n_2} B_{\omega_2} } * \cdots * \delta_{ N_{\omega_1}^{-n_1} N_{\omega_2}^{-n_2} \cdots N_{\omega_k}^{-n_k} B_{\omega_k} } * \cdots.
\end{equation}
For $q \ge 1$, we write
$$\mu_{\omega,\{n_k\},q} = \delta_{ N_{\omega_1}^{-n_1} B_{\omega_1} } * \delta_{ N_{\omega_1}^{-n_1} N_{\omega_2}^{-n_2} B_{\omega_2} } * \cdots * \delta_{ N_{\omega_1}^{-n_1} N_{\omega_2}^{-n_2} \cdots N_{\omega_q}^{-n_q} B_{\omega_q} }.$$
Note that $\mu_{\omega,\{n_k\},q}$ converges weakly to $\mu_{\omega,\{n_k\}}$ as $q$ approaches the infinity.

\begin{lemma}\label{lemma-uniform-approximation}
  For $f \in C_b(\R)$ and $\ep>0$, there exists $q_0 \ge 1$ such that
  \begin{equation*}
    \left| \int_{\R} f(x) \D\mu_{\omega,\{n_k\},q_0}(x) - \int_{\R} f(x) \D\mu_{\omega,\{n_k\}}(x) \right| < \ep
  \end{equation*}
  for all $\omega \in \Omega$ and all sequences of positive integers $\{n_k\}$.
\end{lemma}
\begin{proof}
  Let $h = \max\set{ |b|: b \in B_j,\; 1 \le j \le m }$.
  For $f \in C_b(\R)$ and $\ep>0$, since $f$ is uniformly continuous on $[-(h+1), h+1]$, there exists $0< \gamma <1$ such that for all $|x|,|y|\le h+1$ with $|x-y| < \gamma$ we have that
  \begin{equation}\label{uniform-continuous}
    |f(x)-f(y)|< \ep.
  \end{equation}
  Choose a sufficiently large integer $q_0$ such that $2^{-q_0} h < \gamma$.

  Fix $\omega \in \Omega$ and a sequence of positive integers $\{n_k\}$.
  For every sequence $\{ b_{\omega_k} \}_{k=1}^\f$ where $b_{\omega_k} \in B_{\omega_k}$ for each $k \ge 1$, noting that $|N_j| \ge 2$ for $1 \le j \le m$, we have that
  $$\left| \sum_{k=1}^{q} \frac{ b_{\omega_k} }{ N_{\omega_1}^{n_1} N_{\omega_2}^{n_2} \cdots N_{\omega_k}^{n_k} } \right| \le \sum_{k=1}^{q} \frac{h}{2^k} < h.$$
  Thus, we have that $$\mathrm{spt}(\mu_{\omega,\{n_k\}, q}) \subset [-h,h].$$
  The infinite convolution $\mu_{\omega,\{n_k\}}$ may be written as $\mu_{\omega,\{n_k\}} = \mu_{\omega,\{n_k\},q} * \mu_{\omega,\{n_k\},>q}$,
  where $$ \mu_{\omega,\{n_k\},>q} = \delta_{ N_{\omega_1}^{-n_1} N_{\omega_2}^{-n_2} \cdots N_{\omega_{q+1}}^{-n_{q+1}} B_{\omega_{q+1}} } * \delta_{ N_{\omega_1}^{-n_1} N_{\omega_2}^{-n_2} \cdots N_{\omega_{q+2}}^{-n_{q+2}} B_{\omega_{q+2}} } * \cdots $$
  is the tail of infinite convolution.
  It is easy to check that $$\mathrm{spt}( \mu_{\omega,\{n_k\},>q} ) \subset [-2^{-q} h, 2^{-q}h ].$$
  Note that
  \begin{align*}
    \int_{\R} f(x) \D \mu_{\omega,\{n_k\}}(x)
    & = \int_{\R} f(x) \D \mu_{\omega,\{n_k\},q} * \mu_{\omega,\{n_k\},>q}(x) \\
    & = \int_{\R^2} f(x+y) \D \mu_{\omega,\{n_k\},q} \times \mu_{\omega,\{n_k\},>q}(x,y) \\
    & = \int_{\R} \int_{\R} f(x+y) \D \mu_{\omega,\{n_k\},>q}(y) \D \mu_{\omega,\{n_k\},q}(x).
  \end{align*}
  Thus, by \eqref{uniform-continuous}, we have that
  \begin{align*}
    &~\left| \int_{\R} f(x) \D\mu_{\omega,\{n_k\},q_0}(x) - \int_{\R} f(x) \D\mu_{\omega,\{n_k\}}(x) \right| \\
    = & ~\left| \int_{\R} \int_{\R} \big( f(x) - f(x+y) \big) \D \mu_{\omega,\{n_k\},>q_0}(y) \D \mu_{\omega,\{n_k\},q_0}(x) \right| \\
    \le & ~\int_{-h}^{h} \int_{-2^{-q_0}h}^{2^{-q_0} h} \left|  f(x) - f(x+y) \right| \D \mu_{\omega,\{n_k\},>q_0}(y) \D \mu_{\omega,\{n_k\},q_0}(x) \\
    < & ~\ep.
  \end{align*}
  This completes the proof.
\end{proof}

\begin{proposition}\label{prop-unbounded}
  Let $\{(N_j,B_j)\}_{j=1}^m$ be finitely many admissible pairs, and let $\mu_{\omega,\{n_k\}}$ be defined by \eqref{mu-omega-n-k}.
  If the sequence $\{n_k\}$ is unbounded, then all infinite convolutions $\mu_{\omega,\{n_k\}}$ for $\omega \in \Omega$ are spectral measures.
\end{proposition}
\begin{proof}
  Fix $\omega \in \Omega$ and write $\mu = \mu_{\omega,\{n_k\}}$.
  Then the measure $\mu$ is the infinite convolution generated by the sequence of admissible pairs $\{ ( (N_{\omega_k})^{n_k}, B_{\omega_k}) \}_{k=1}^\f$.
  Recall the notation $\nu_{>k}$ defined in \eqref{nu-large-than-n} for the infinite convolution $\mu$.
  We have that $$\nu_{>k} = \delta_{ N_{\omega_{k+1}}^{-n_{k+1}} B_{\omega_{k+1}} } * \delta_{ N_{\omega_{k+1}}^{-n_{k+1}} N_{\omega_{k+2}}^{-n_{k+2}}  B_{\omega_{k+2}} }* \cdots.$$
  Let $h = \max\set{ |b|: b \in B_j,\; 1 \le j \le m }$, and it is clear that
  $$\mathrm{spt}(\nu_{>k}) \subset [ -2^{-n_{k+1}+1} h,  2^{-n_{k+1}+1} h ]. $$

  Since the sequence $\{n_k\}$ is unbounded, we may find a subsequence $\{n_{k_j}\}$ such that $$\lim_{j \to \f} n_{k_j}= \f.$$
  Consider the sequence $\{\nu_{>k_j -1}\}$, and we have that $\mathrm{spt}(\nu_{>k_j -1}) \subset [ -2^{-n_{k_j}+1} h,  2^{-n_{k_j}+1} h ].$
  Thus, the sequence $\{ \nu_{>k_j-1} \}$ converges weakly to $\delta_0$, the Dirac measure concentrated on $0$.
  Obviously, we have that $\mathcal{Z}(\delta_0) =\emptyset$.
  It follows from Theorem \ref{main-spectrality} that $\mu = \mu_{\omega,\{n_k\}}$ is a spectral measure.
\end{proof}

We denote by $\sigma$ the left shift on the symbolic space $\Omega$, that is, $$\sigma(\omega) = \omega_2 \omega_3 \omega_4 \cdots $$
for $\omega = (\omega_k)_{k=1}^\f \in \Omega$.

\begin{lemma}\label{lemma-4-3}
  Suppose that the Bernoulli measure $\PP$ on $\Omega$ is associated with a positive probability vector.
  Then for $\PP$-a.e. $\omega \in \Omega$, there exists a subsequence $\set{k_j}$ such that $\{\sigma^{k_j}(\omega)\}$ converges to $(12\cdots m)^\f$ in $\Omega$.
\end{lemma}
\begin{proof}
  Let $(p_1,p_2,\ldots,p_m)$ be the positive probability vector associated with the Bernoulli measure $\PP$.
  Since the left shift $\sigma$ is ergodic with respect to the Bernoulli measure $\PP$ on $\Omega$ \cite{Walters-1982}, by Birkhoff ergodic theorem, there exists a full measure subset $\Omega_0 \subset \Omega$ such that
  for any $\omega \in \Omega_0$ and for any finite word $i_1 i_2 \cdots i_q $, where $i_j \in \{1,2,\ldots,m\}$ for $1 \le j \le q$, we have that $$\lim_{n \to \f} \frac{\# \set{ 1 \le k \le n: \omega_k \omega_{k+1} \cdots \omega_{k+q-1} = i_1 i_2 \cdots i_q } }{n} = p_{i_1} p_{i_2} \cdots p_{i_q}.$$
  It follows that for $\omega \in \Omega_0$ and $q\ge 1$,
  \begin{equation}\label{positive-frequency}
    \lim_{n \to \f} \frac{\# \set{ 1 \le k \le n: \omega_k \omega_{k+1} \cdots \omega_{k+qm-1} = (1 2 \cdots m)^q } }{n} = (p_1 p_2 \cdots p_m)^q >0.
  \end{equation}

  Fix $\omega \in \Omega_0$. By \eqref{positive-frequency}, we first choose $k_1 \ge 1$ such that $\omega_{k_1 +1} \omega_{k_1+2} \cdots \omega_{k_1 + m} = 12 \cdots m$.
  Assume that we have chosen $k_1< k_2 < \cdots < k_q$ such that $\omega_{k_j +1} \omega_{k_j+2} \cdots \omega_{k_j + jm} = (12 \cdots m)^j$ for $1\le j \le q$.
  Then, by \eqref{positive-frequency}, we can find $k_{q+1} > k_q$ such that $$\omega_{k_{q+1} +1} \omega_{k_{q+1}+2} \cdots \omega_{k_{q+1} + (q+1)m} = (12 \cdots m)^{q+1}.$$
  Thus, we recursively find a subsequence $\set{k_j}$ such that for all $j \ge 1$,
  $$\omega_{k_j +1} \omega_{k_j+2} \cdots \omega_{k_j + jm} = (12 \cdots m)^j.$$
  It follows that $\{\sigma^{k_j}(\omega)\}$ converges to $(12\cdots m)^\f$ in $\Omega$.
  The proof is completed.
\end{proof}

\begin{proposition}\label{prop-bounded}
  Let $\{(N_j,B_j)\}_{j=1}^m$ be finitely many admissible pairs, and let $\mu_{\omega,\{n_k\}}$ be defined by \eqref{mu-omega-n-k}.
  If the sequence $\{ n_k \}$ is bounded, then for every Bernoulli measure $\PP$ on $\Omega$, the infinite convolution $\mu_{\omega,\{n_k\}}$ is a spectral measure for $\PP$-a.e. $\omega \in \Omega$.
\end{proposition}
\begin{proof}
  The probability vector associated with the Bernoulli measure $\PP$ is denoted by $(p_1, p_2, \ldots, p_m)$.
  By rearranging the symbols, there exists $1\le m' \le m$ such that $p_j>0$ for $1\le j \le m'$ and $p_{m'+1} = p_{m'+2}= \cdots = p_m =0$.
  Let $\Omega'= \{1,2,\ldots, m'\}^\N$. Then we have that $\PP(\Omega') =1$.
  Let $\PP'$ be the restriction of $\PP$ on $\Omega'$. In fact, $\PP'$ is the Bernoulli measure on $\Omega'$ associated with the positive probability vector $(p_1, p_2, \ldots, p_{m'})$.
  It suffices to show that the infinite convolution $\mu_{\omega,\{n_k\}}$ is a spectral measure for $\PP'$-a.e. $\omega \in \Omega'$.
  Thus, we assume that the probability vector $(p_1,p_2,\ldots,p_m)$ is positive.

  Let $$d = \gcd\left( \bigcup_{j=1}^m (B_j - B_j) \right).$$
  For $1 \le j \le m$, we define $B'_j = (B_j - b_j)/d$ for some $b_j \in B_j$.
  Then we have that $$ \gcd\left( \bigcup_{j=1}^m (B'_j - B'_j) \right) =1. $$
  We write $$\mu'_{\omega,\{n_k\}} = \delta_{ N_{\omega_1}^{-n_1} B'_{\omega_1} } * \delta_{ N_{\omega_1}^{-n_1} N_{\omega_2}^{-n_2} B'_{\omega_2} } * \cdots * \delta_{ N_{\omega_1}^{-n_1} N_{\omega_2}^{-n_2} \cdots N_{\omega_k}^{-n_k} B'_{\omega_k} } * \cdots.$$
  It is easy to check that $$\mu_{\omega,\{n_k\}} = \mu'_{\omega,\{n_k\}} \circ T_{d,b_\omega}^{-1}, $$
  where $$b_\omega = \sum_{k=1}^{\f} \frac{ b_{\omega_k} }{ N_{\omega_1}^{n_1} N_{\omega_2}^{n_2} \cdots N_{\omega_k}^{n_k} },$$
  and the function $T_{d,b_{\omega}}$ is defined in \eqref{linear-transformation}.
  By Lemma \ref{spectrality-invariant}, it suffices to show that $\mu'_{\omega,\{n_k\}}$ is a spectral measure for $\PP$-a.e. $\omega \in \Omega$.
  Note by Lemma \ref{admissible-pair-lemma} (iii) and (iv) that $(N_j, B'_j)$ is also an admissible pair for each $1 \le j \le m$.
  So, in the following, we also assume that
  \begin{equation}\label{prime-condition}
    \gcd\left( \bigcup_{j=1}^m (B_j - B_j) \right)=1.
  \end{equation}

  Let $\eta=(\eta_k)_{k=1}^\f = (12\cdots m)^\f$.
  The measure $\mu_{\eta,\{n_k\}}$ is the infinite convolution generated by the sequence of admissible pairs $\{ ( (N_{\eta_k})^{n_k}, B_{\eta_k}) \}_{k=1}^\f$.
  Since the sequence $\{ n_k \}$ is bounded, the sequence $\{ ( (N_{\eta_k})^{n_k}, B_{\eta_k}) \}_{k=1}^\f$ is chosen from a finite set of admissible pairs.
  By \eqref{prime-condition} and Corollary \ref{coro-finite-admissible-pair}, we have that
  \begin{equation}\label{mu-eta-empty}
    \mathcal{Z}(\mu_{\eta,\{n_k\}}) = \emptyset,
  \end{equation}
  for all bounded sequences $\{n_k\}$.

  By Lemma \ref{lemma-4-3}, there exists a full measure subset $\Omega_0 \subset \Omega$ such that for $\omega \in \Omega_0$ we have that $\{\sigma^{k_j}(\omega)\}$ converges to $\eta$ for some subsequence $\{k_j\}$.
  Fix $\omega \in \Omega_0$ and such a subsequence $\set{k_j}$.
  Next, we show that $\mu_{\omega,\{n_k\}}$ is a spectral measure.

  Write $\mu = \mu_{\omega,\{n_k\}}$.
  Then the measure $\mu$ is the infinite convolution generated by the sequence of admissible pairs $\{ ( (N_{\omega_k})^{n_k}, B_{\omega_k}) \}_{k=1}^\f$.
  Recall the notation $\nu_{>k}$ defined in \eqref{nu-large-than-n} for the infinite convolution $\mu$.
  For $k \ge 1$, we have that $$\nu_{>k} = \mu_{\sigma^k(\omega), \{ n_{k+\ell} \}_{\ell=1}^\f}.$$
  Let $M=\max\{n_k: k \ge 1\}$, and let $\Sigma$ be the symbolic space over the alphabet $\{1,2,\ldots, M\}$.
  By the compactness of $\Omega \times \Sigma$, the sequence $$\set{ \big( \sigma^{k_j}(\omega), \{ n_{k_j+\ell} \}_{\ell=1}^\f \big) }_{j=1}^\f$$ has a convergent subsequence in $\Omega \times \Sigma$.
  By taking the subsequence, we assume that $\set{ \big( \sigma^{k_j}(\omega),\{ n_{k_j+\ell} \}_{\ell=1}^\f \big) }_{j=1}^\f$ converges to $(\eta, \{ m_k\})$ for some sequence $\{m_k\} \in \Sigma$.

  For $f \in C_b(\R)$ and $\ep >0$, by Lemma \ref{lemma-uniform-approximation}, there exists $q_0 \ge 1$ such that
  $$\left| \int_{\R} f(x) \D \mu_{ \sigma^{k_j}(\omega), \{ n_{k_j+\ell} \}_{\ell=1}^\f, q_0 }(x) -  \int_{\R} f(x) \D \mu_{ \sigma^{k_j}(\omega), \{ n_{k_j+\ell} \}_{\ell=1}^\f }(x) \right| < \frac{\ep}{2}$$
  for all $j \ge 1$, and
  $$ \left| \int_{\R} f(x) \D \mu_{\eta, \{ m_k \}, q_0}(x) -  \int_{\R} f(x) \D \mu_{\eta, \{ m_k \}}(x) \right| < \frac{\ep}{2}.$$
  Since $\set{ \big( \sigma^{k_j}(\omega),\{ n_{k_j+\ell} \}_{\ell=1}^\f \big) }_{j=1}^\f$ converges to $(\eta, \{ m_k\})$,
  there exists $j_0 \ge 1$ such that for $j \ge j_0$, we have that
  $$\mu_{ \sigma^{k_j}(\omega), \{ n_{k_j+\ell} \}_{\ell=1}^\f, q_0 } = \mu_{\eta, \{ m_k \}, q_0}. $$
  Thus, for $j \ge j_0$,
  $$\left| \int_{\R} f(x) \D \mu_{ \sigma^{k_j}(\omega), \{ n_{k_j+\ell} \}_{\ell=1}^\f }(x) - \int_{\R} f(x) \D \mu_{\eta, \{ m_k \}}(x) \right| < \ep. $$
  This implies that the sequence $\{\nu_{>k_j} = \mu_{ \sigma^{k_j}(\omega), \{ n_{k_j+\ell} \}_{\ell=1}^\f } \}$ converges weakly to $\mu_{ \eta, \{ m_k \} }$.
  By \eqref{mu-eta-empty}, we have that $\mathcal{Z}(\mu_{ \eta, \{ m_k \} }) = \emptyset$.
  It follows from Theorem \ref{main-spectrality} that $\mu = \mu_{\omega,\{n_k\}}$ is a spectral measure.
\end{proof}

\begin{proof}[Proof of Theorem \ref{general-theorem}]
  It follows from Proposition \ref{prop-unbounded} and Proposition \ref{prop-bounded}.
\end{proof}

\section{Spectrality of special random convolutions}\label{section-special-case}

Recall that $t \ge 1$ is an integer, $N,p \ge 2$ are integers with $\gcd(N,p)=1$, and
$$N_1 = N_2 = tN,\; B_1=\set{0,1,\ldots, N - 1},\; B_2 = p \set{0,1,\ldots, N - 1}.$$
Let $L = \{ 0, t, 2t, \ldots, (N-1)t \}$.
It is straightforward to verify that these two matrices $$ \left( \frac{1}{\sqrt{N}} e^{-2\pi i \frac{b\ell}{tN}} \right)_{b \in B_1, \ell \in L},\; \left( \frac{1}{\sqrt{N}} e^{-2\pi i \frac{b\ell}{tN}} \right)_{b \in B_2, \ell \in L} $$
are unitary.
This implies that $(N_1, B_1)$ and $(N_2, B_2)$ are admissible pairs.

For $\omega \in \Omega = \{1,2\}^\N$, we define
$$ \mu_\omega = \delta_{(tN)^{-1} B_{\omega_1}} * \delta_{(tN)^{-2} B_{\omega_2}}* \cdots * \delta_{(tN)^{-k} B_{\omega_k}} * \cdots.$$
It follows that $$\wh{\mu}_\omega(\xi) = \prod_{k=1}^{\f} M_{B_{\omega_k}} \left( \frac{\xi}{(tN)^k} \right).$$
As we have done in the proof of Lemma \ref{lemma-zero-set}, we conclude that
\begin{equation}\label{zero-set-random-convolution}
  \mathcal{O}(\widehat{\mu}_\omega) = \bigcup_{k=1}^\f (tN)^k \mathcal{O}\big( M_{B_{\omega_k}} \big).
\end{equation}

\begin{lemma}\label{lemma-mu-eta-emptyset}
  For $\eta = 12^\f$, we have that $\mathcal{Z}(\mu_\eta) = \emptyset$.
\end{lemma}
\begin{proof}
  By simple calculation, we have that $$\mathcal{O}(M_{B_1}) = \frac{ \Z \sm N\Z }{N},\; \mathcal{O}(M_{B_2}) = \frac{ \Z \sm N\Z }{pN}. $$
  It follows from \eqref{zero-set-random-convolution} that
  \begin{align*}
    \mathcal{O}(\wh{\mu}_\eta) & = \big( tN\mathcal{O}(M_{B_1}) \big) \cup \bigcup_{k=2}^\f (tN)^k \mathcal{O}(M_{B_2}) \\
    & = \big( t(\Z \setminus N\Z) \big) \cup \bigcup_{k=1}^\f \frac{t^{k+1}N^k (\Z \setminus N\Z)}{p} \\
    & \subseteq \left( \Z \cup \frac{N\Z}{p} \right) \setminus \{0\}.
  \end{align*}

  Suppose that $\mathcal{Z}(\mu_\eta) \ne \emptyset$. Take $\xi_0 \in \mathcal{Z}(\mu_\eta) \cap [0,1)$.
  Then we have that $\widehat{\mu}_\eta(\xi_0 + k) = 0$ for all $k \in \Z$.
  It follows that $\xi_0, \xi_0 +1 \in \mathcal{O}(\widehat{\mu}_\eta)$.
  Thus, there exists $k_1,k_2 \in \Z$ such that $$ \xi_0 = \frac{k_1 N}{p},\; \xi_0 + 1 = \frac{k_2 N}{p}.$$
  So we have that $$p = (k_2 -k_1)N.$$
  This contradicts with $\gcd(N,p)=1$.
  Therefore, we conclude that $\mathcal{Z}(\mu_\eta) = \emptyset$.
\end{proof}

\begin{proposition}\label{spectrality-1-occur-infinite}
  For $\omega \in \Omega$, if $\omega = 2^\f$ or the symbol ``1" occurs infinitely many times in $\omega$, then the infinite convolution $\mu_\omega$ is a spectral measure.
\end{proposition}
\begin{proof}
For $\omega=2^\f$, the infinite convolution $\mu_{\omega}$ is the self-similar measure generated by the admissible pair $(N_2,B_2)$, and it follows from the classical result by {\L}aba and Wang \cite{Laba-Wang-2002} that it is a spectral measure.
In the following, we assume that the symbol ``1" occurs infinitely many times in $\omega$.

Write $\mu = \mu_{\omega}$.
Then the measure $\mu$ is the infinite convolution generated by the sequence of admissible pairs $\{ (N_{\omega_k}, B_{\omega_k}) \}_{k=1}^\f$.
Recall the notation $\nu_{>n}$ defined in \eqref{nu-large-than-n} for the infinite convolution $\mu$.
Then we have that $\nu_{>n} = \mu_{\sigma^n(\omega)}$.
The proof is divided into two cases.

\textbf{Case (i)}: there exists $\ell \in \N$ such that the sequence $\omega$ does not contain the finite word $2^\ell$.
By the compactness of $\Omega$, there exists a subsequence $\set{n_j}$ such that $\{\sigma^{n_j}(\omega)\}$ converges to $\zeta$ in $\Omega$ for some $\zeta \in \Omega$.
As in the proof of Proposition \ref{prop-bounded}, we have that the sequence $\{ \nu_{>n_j}=\mu_{\sigma^{n_j}(\omega)} \}$ converges weakly to $\mu_\zeta$.
Note that the sequence $\sigma^{n_j}(\omega)$ does not contain the finite word $2^\ell$ for every $j \ge 1$.
Thus the sequence $\zeta$ also does not contain the finite word $2^\ell$.
This means that the symbol ``1'' occurs infinitely many times in $\zeta$.
It follows from Corollary \ref{coro-finite-admissible-pair} that $\mcal{Z}(\mu_\zeta) = \emptyset$.
Therefore, by Theorem \ref{main-spectrality}, we conclude that $\mu=\mu_\omega$ is a spectral measure.

\textbf{Case (ii)}: for every $\ell \ge 1$, the sequence $\omega$ contains the finite word $2^\ell$.
Note that the symbol ``1" occurs infinitely many times in $\omega$.
Let $$ \{ k_1 < k_2 < \cdots < k_j < \cdots \} = \set{k \ge 1: \omega_k =1}. $$
Then we have that $$\limsup_{j \to \f} (k_{j+1} - k_j)  = \f.$$
So we may find a subsequence $\{n_j\}$ such that $$\lim_{j \to \f} (k_{n_j+1} - k_{n_j})  = \f.$$
Note that the sequence $\sigma^{k_{n_j} -1} (\omega)$ begins with the finite word $12^{k_{n_j+1} - k_{n_j} -1}$.
Thus, we have that $\{\sigma^{k_{n_j} -1} (\omega)\}$ converges to $\eta = 12^\f$ in $\Omega$.
As in the proof of Proposition \ref{prop-bounded}, we have that the sequence $\{ \nu_{> k_{n_j} -1} =\mu_{\sigma^{k_{n_j} -1}(\omega)} \}$ converges weakly to $\mu_\eta$.
By Lemma \ref{lemma-mu-eta-emptyset}, we have that $\mathcal{Z}(\mu_\eta) = \emptyset$.
Therefore, by Theorem \ref{main-spectrality}, we conclude that $\mu=\mu_\omega$ is a spectral measure.
\end{proof}

Recall that we write
$$\mu_{N,B} = \delta_{N^{-1}B} * \delta_{N^{-2} B} * \cdots * \delta_{N^{-k} B} * \cdots $$
for the self-similar measure.

\begin{lemma}\label{lemma-no-overlap}
  Suppose that $t \ge 2$ and let $\mu= \mu_{tN,B_2}$.
  Then we have that for all $j\in \Z \setminus \{0\}$, $$\mu\big( \mathrm{spt}(\mu) +j \big)=0.$$
\end{lemma}
\begin{proof}
  Note that $B_2 = p \{ 0,1,\ldots, N-1 \}$.
  We have that $$\mathrm{spt}(\mu) = \left\{ p\sum_{k=1}^{\f}\frac{\ep_k}{(tN)^k} : \ep_k \in \{0,1,\ldots,N-1\} \text{ for } k\in \N \right\}.$$
  Let $\Sigma = \{ 0,1,2,\ldots, N -1\}^\N$ and let $\PP$ be the uniform Bernoulli measure on $\Sigma$.
  Consider the coding mapping $\pi_{tN}: \Sigma \to \mathrm{spt}(\mu)$ defined by $$\pi_{tN}\big( (\ep_k) \big) = p\sum_{k=1}^{\f}\frac{\ep_k}{(tN)^k}.$$
  Then we have that $\mu = \PP \circ \pi_{tN}^{-1}$.

  Suppose that there exists $j \in \Z\setminus \{0\}$ such that $\mu(\mathrm{spt}(\mu)+j)>0$.
  Let $$\Sigma'=\set{ (\ep_k) \in \Sigma: p\sum_{k=1}^{\f} \frac{\ep_k}{(tN)^k} \in  \mathrm{spt}(\mu) +j }.$$
  Then we have that $\PP(\Sigma') >0$.
  For $(\ep_k) \in \Sigma'$, we have that
  $$ p\sum_{k=1}^{\f} \frac{\ep_k}{(tN)^k} = j+ p\sum_{k=1}^{\f} \frac{\eta_k}{(tN)^k}$$ for some $(\eta_k) \in \Sigma$.
  That is, $$\frac{j}{p} + \frac{N-1}{tN-1} = \sum_{k=1}^{\f} \frac{\ep_k + (N-1-\eta_k)}{(tN)^k}.$$
  Note that $\ep_k + (N-1-\eta_k) \in \set{0,1,2,\ldots, 2N-2}$ for each $k \ge 1$.
  Thus, there exists a unique sequence $(\zeta_k)\in \set{0,1,2,\ldots, 2N-2}^\N$ such that $$ \frac{j}{p} + \frac{N-1}{tN-1} = \sum_{k=1}^{\f} \frac{\zeta_k}{(tN)^k}. $$
  It follows that for $k \ge 1$, $$ \zeta_k = \ep_k + (N-1 -\eta_k),$$
  i.e., $$\ep_k = \zeta_k + \eta_k - N + 1.$$
  Thus, for each $k \ge 1$, we have that
  $$ \ep_k \in
  \begin{cases}
    \mathcal{A}_{\zeta_k} = \set{0, 1,\ldots, \zeta_k}, & \mbox{if } \zeta_k < N-1; \\
    \mathcal{A}_{\zeta_k} = \set{0,1,\ldots,N-1}, & \mbox{if } \zeta_k = N-1; \\
    \mathcal{A}_{\zeta_k} = \set{\zeta_k -N+1, \zeta_k -N +2, \ldots , N-1}, & \mbox{if } \zeta_k > N-1.
  \end{cases}
  $$
  So we get that $$\Sigma' = \set{ (\ep_k) \in \Sigma: \ep_k \in \mathcal{A}_{\zeta_k} \text{ for } k \in \N }.$$
  It follows that $$\PP(\Sigma') = \lim_{n \to \f} \prod_{k=1}^n \frac{\#\mathcal{A}_{\zeta_k}}{N}. $$
  Note that $\PP(\Sigma') >0$.
  Thus the sequence $(\zeta_k)$ ends with $(N-1)^\f$.

  Let $k_0 = \min\{ k \in \N: \zeta_n =N-1 \text{ for all } n > k \}$. Then $k_0 \ge 1$ and $\zeta_{k_0} \ne N -1$.
  Write $$x= \frac{j}{p} + \frac{N-1}{tN-1},\; x_k = \sum_{\ell=1}^{\f} \frac{\zeta_{k+\ell}}{(tN)^\ell}. $$
  Note that $$(tN)^k \frac{N-1}{tN-1} \pmod{1} \equiv \frac{N-1}{tN-1}.$$
  Thus we have that $$(tN)^k x \pmod{1} \equiv \frac{(tN)^k j}{p} + \frac{N-1}{tN-1} \pmod{1} \equiv x_k .$$
  Since $x_{k_0}=(N-1)/(tN-1)$, we must have that $$  p \mid (tN)^{k_0}j. $$
  Moreover, we have that $$ p \nmid (tN)^{k_0-1}j. $$

  Let $q \in \{0,1,\ldots,p-1\}$ satisfy $q \equiv (tN)^{k_0-1}j \pmod{p}$.
  Since $ p \mid (tN)^{k_0}j$, we have that $p \mid qtN$.
  Noting that $\gcd(N,p)=1$, we obtain that $p \mid qt$.
  It follows that $$N \le \frac{q t N}{p} \le (t-1)N. $$
  Thus, we have that
  $$\frac{q}{p} + \frac{N-1}{tN-1} \le \frac{(t-1)N}{tN} +\frac{N-1}{tN-1} < 1,$$
  and $$\frac{q}{p} + \frac{N-1}{tN-1} \ge \frac{N}{tN} +\frac{N-1}{tN-1} > \frac{2N-1}{tN}. $$
  On the one hand, $$ x_{k_0-1} \equiv (tN)^{k_0 -1} x \pmod{1} = \frac{q}{p} + \frac{N-1}{tN-1} > \frac{2N-1}{tN}.$$
  On the other hand, $$x_{k_0-1} = \sum_{\ell=1}^{\f} \frac{\zeta_{k_0-1+\ell}}{(tN)^\ell} < \frac{\zeta_{k_0}+1}{tN} \le \frac{2N-1}{tN}.$$
  This is a contradiction.

  Therefore, we obtain that $\mu\big( \mathrm{spt}(\mu) +j \big)=0$ for all $j \in \Z \setminus \{0\}$.
\end{proof}

\begin{proof}[Proof of Theorem \ref{special-case}]
  By Proposition \ref{spectrality-1-occur-infinite}, it remains to consider the infinite convolution $\mu_\omega$ for $\omega \in \Omega$ in which the symbol ``1'' occurs only finitely many times.
  In the following, we fix $\omega\in \Omega$ satisfying $$\omega_{k_0}=1 \text{ and } \omega_{k_0+1} \omega_{k_0+2} \omega_{k_0+3} \cdots = 2^\f$$ for some $k_0 \in \N$.

  (i) For the case that $t=1$,
  let $$ B = N^{k_0-1} B_{\omega_1} + N^{k_0-2} B_{\omega_2} + \cdots + N B_{\omega_{k_0-1}}+ B_{\omega_{k_0}}, $$
  and define $$\widetilde{\mu}_\omega = \delta_{B} * \mu_{N, B_2}.$$
  Note that $$\mu_\omega = \delta_{N^{-1} B_{\omega_1} } * \cdots * \delta_{N^{-k_0} B_{\omega_{k_0}} } * \delta_{N^{-k_0-1} B_2 } * \delta_{N^{-k_0-2} B_2 } * \cdots. $$
  We have that $\widetilde{\mu}_\omega = \mu_{\omega} \circ T_{N^{k_0},0}^{-1}$, where the function $T_{N^{k_0},0}$ is defined in \eqref{linear-transformation}.
  By Lemma \ref{spectrality-invariant}, $\mu_\omega$ is a spectral measure if and only if $\widetilde{\mu}_\omega$ is a spectral measure.

  Note that $$ \mu_{N,B_2} = \frac{1}{p} \mathcal{L}|_{[0,p]}, $$ where $\mathcal{L}|_{[0,p]}$ denotes the restriction of Lebesgue measure on the interval $[0,p]$.
  Thus, the measure $\widetilde{\mu}_\omega$ may be written as $$ \widetilde{\mu}_\omega = \frac{1}{\# B}\sum_{b \in B} \delta_{b} * \mu_{N, B_2} = \frac{1}{p \cdot \# B} \sum_{b \in B} \mathcal{L}|_{[b,b+p]}. $$
  Noting that $\{0,1,\ldots,N-1\} \subset B$ and $p \ge 2$, it is easy to check that $\widetilde{\mu}_\omega$ is not uniformly distributed on its support.
  However, an absolutely continuous spectral measure must be uniform on its support \cite{Dutkay-Lai-2014}.
  As a consequence, $\widetilde{\mu}_\omega$ is not a spectral measure.
  It follows that $\mu_\omega$ is not a spectral measure.

  (ii) For the case that $t \ge 2$, the measure $\mu=\mu_\omega$ is the infinite convolution generated by the sequence of admissible pairs $\set{(N_{\omega_k}, B_{\omega_k})}_{k=1}^\f$.
  Recall the notation $\nu_{>n}$ defined in \eqref{nu-large-than-n} for the infinite convolution $\mu$.
  Then we have that $\nu_{>n} = \mu_{\sigma^n(\omega)} = \mu_{tN,B_2}$ for all $n \ge k_0$.
  By Lemma \ref{lemma-no-overlap} and Theorem \ref{theorem-IPZS-empty-1}, we have that $\mathcal{Z}(\mu_{tN,B_2}) = \emptyset$.
  It follows from Theorem \ref{main-spectrality} that $\mu=\mu_\omega$ is a spectral measure.
\end{proof}

\section{Examples}\label{section-example}

In this section, we give some examples to illustrate our results.
\begin{example}
  Let $N_1 = N_2 =2$, $B_1 = \set{0,1}$, and $B_2 = \set{0,3}$.
  Example 1.8 in \cite{Dutkay-Lai-2017} showed that the random convolution $\mu_\eta$ is not a spectral measure for $\eta = 1 2^\f$ because $\mu_\eta$ is not uniformly distributed on its support.
  By Theorem \ref{special-case}, for $\omega \in \{1,2\}^{\N}$, we have that $\mu_\omega$ is a spectral measure if and only if $\omega = 2^\f$ or the symbol ``1" occurs infinitely many times in $\omega$.
  In particular, for $$\omega = 1\; 2\; 11\; 12\; 21\; 22\; 111\; 112 \;\cdots$$
  which enumerates all finite words in lexicographical order, we have that $\mu_\omega$ is a spectral measure.
\end{example}

Next, we construct some new examples of spectral measures by applying Theorem \ref{main-spectrality} and Theorem \ref{theorem-IPZS-empty-1}.
The spectrality of infinite convolutions with three elements in digit sets has been studied in \cite{An-He-He-2019,Fu-Wen-2017}.
However, the following example cannot be deduced directly from these known results.

\begin{example}\label{new-example}
  Let $N > 3$ be an integer with $3\mid N$. Choose a sequence of integers $\{b_j\}_{j=1}^\f$ such that $3 \mid b_j$ and
  \begin{equation}\label{example-limit}
    \lim_{j\to \f} \frac{b_j}{N^j} = n_0 + \frac{5}{6}
  \end{equation}
  for some $n_0 \in \Z$.
For $k \ge 1$, we define
\begin{equation*}
B_k =
  \begin{dcases}
    \{2,4 , b_j \}, & \mbox{if } k = j(j+1)/2 \text{ for some } j \ge 1; \\
    \{0,2,4\}, & \mbox{otherwise}.
  \end{dcases}
\end{equation*}
Then the infinite convolution
\begin{equation}\label{example-infinite-convolution}
  \mu = \delta_{N^{-1}B_1} * \delta_{N^{-2}B_2} * \cdots * \delta_{N^{-k} B_k} * \cdots
\end{equation}
is a spectral measure.

Since $B_k \equiv \{0,1,2\} \pmod{3}$ for $k \ge 1$, the discrete measure $\delta_{N^{-1}B_k}$ admits a spectrum $L =\{0,N/3,2N/3\}$ for all $k \ge 1$.
Thus, $\{(N,B_k)\}_{k=1}^\f$ is a sequence of admissible pairs.
Let $$ M = \sup_{j \ge 1} \frac{|b_j|}{N^j}. $$
It is easy to check that $$ \sum_{k=1}^\f \frac{\max\{ |b|: b \in B_k\}}{N^k} \le 1+ \sum_{j=1}^{\f} \frac{|b_j|}{N^{\frac{j(j+1)}{2}}} \le 1 + \sum_{j=1}^{\f} \frac{M}{N^{\frac{j(j-1)}{2}}} \le  1+2M. $$
Thus the infinite convolution $\mu$ defined in \eqref{example-infinite-convolution} exists.

Let $n_j = j(j+1)/2$ for $j \ge 1$.
Then we have that
  \begin{align*}
    \nu_{>n_j} & = \delta_{N^{-1} B_{n_j+1}} * \delta_{N^{-2} B_{n_j+2}} * \cdots * \delta_{N^{-k} B_{n_j+k}} * \cdots \\
    & = \delta_{N^{-1} \{0,2,4\} } * \cdots * \delta_{N^{-j} \{0,2,4\} } * \delta_{N^{-(j+1)} \{2,4,b_{j+1}\} } * \nu_j,
  \end{align*}
where $\nu_j$ is the tail part of infinite convolution.
One can easily check that
\begin{equation}\label{support-nu-j}
  \mathrm{spt}(\nu_j) \subset \left[ -\frac{1+2M}{N^{j+1}}, \frac{1+2M}{N^{j+1}} \right].
\end{equation}
Note that
\begin{align*}
  \nu_{>n_j} & = \frac{2}{3} \delta_{N^{-1} \{0,2,4\} } * \cdots * \delta_{N^{-j} \{0,2,4\} } * \delta_{N^{-(j+1)} \{2,4 \}} * \nu_j \\
  & \quad\quad + \frac{1}{3} \delta_{N^{-1} \{0,2,4\} } * \cdots * \delta_{ N^{-j} \{0,2,4\} } * \delta_{\{b_{j+1}/N^{j+1}\}} * \nu_j.
\end{align*}
Thus, by \eqref{example-limit} and \eqref{support-nu-j}, we have that $\{\nu_{>n_j}\}$ converges weakly to
$$\nu = \frac{2}{3} \mu_{N,\{0,2,4\}} + \frac{1}{3} \delta_{\{n_0+5/6\}} * \mu_{N,\{0,2,4\}}, $$
where $\mu_{N,\{0,2,4\}} = \delta_{N^{-1} \{0,2,4\} } *  \delta_{N^{-2} \{0,2,4\} } * \cdots * \delta_{N^{-k} \{0,2,4\} } * \cdots$.

Since $5/6 \not\in \mathrm{spt}(\mu_{N,\{0,2,4\}})$, there exists $0< \gamma <1/6$ such that
$$(5/6-\gamma, 5/6 + \gamma) \cap \mathrm{spt}(\mu_{6,\{0,2,4\}}) = \emptyset.$$
Let $E=(n_0+5/6 -\gamma, n_0+5/6 + \gamma)$.
Then we have that $\nu(E)>0$ and $\nu(E+k)=0$ for all $k \in \Z \setminus\{0\}$.
By Theorem \ref{theorem-IPZS-empty-1}, we have that $\mathcal{Z}(\nu) = \emptyset$.
It follows from Theorem \ref{main-spectrality} that the infinite convolution $\mu$ defined in \eqref{example-infinite-convolution} is a spectral measure.
\end{example}

Note that the number $5/6$ in \eqref{example-limit} is deliberately chosen.
More generally, if we choose $x_0 \in [0,1) \setminus \mathrm{spt}(\mu_{N,\{0,2,4\}})$, and substitute  \eqref{example-limit} by the condition that
$$ \lim_{j\to \f} \frac{b_j}{N^j} = n_0 + x_0$$ for some $n_0\in \Z$, then as we have done above, the resulting infinite convolution in Example \ref{new-example} is still a spectral measure.

We give another example of spectral measures, whose digit sets are almost consecutive digit sets.

\begin{example}
  Let $N = tp$ where $t,p \ge 2$ are integers.
  Let $0\le x_0 < 1$ satisfy
  $$ x_0 \not \in \left\{ \sum_{k=1}^{\f} \frac{d_k}{N^k}: d_k \in \{0,1,\ldots, p-1\} \text{ for } k \ge 1 \right\}. $$
  For instance, we may choose $x_0=1/2$.
  Choose a sequence of integers $\{b_j\}_{j=1}^\f$ such that $p \mid b_j$ and
  $$\lim_{j \to \f} \frac{b_j}{N^j} = n_0 + x_0$$ for some $n_0 \in \Z$.
For $k \ge 1$, we define
\begin{equation*}
B_k =
  \begin{cases}
    \{1,2,\ldots,p-1 , b_j \}, & \mbox{if } k = j(j+1)/2 \text{ for some } j \ge 1; \\
    \{0,1,\ldots,p-1 \}, & \mbox{otherwise}.
  \end{cases}
\end{equation*}
Then the infinite convolution
\begin{equation}\label{example-2-infinite-convolution}
  \mu = \delta_{N^{-1}B_1} * \delta_{N^{-2}B_2} * \cdots * \delta_{N^{-k} B_k} * \cdots
\end{equation}
is a spectral measure.

  Since $B_k \equiv \{0,1,\ldots, p-1\} \pmod{p}$, the discrete measure $\delta_{N^{-1}B_k}$ admits a spectrum $L =\{0,t,2t,\ldots,(p-1)t\}$ for all $k \ge 1$.
  Thus, $\{(N,B_k)\}_{k=1}^\f$ is a sequence of admissible pairs.
  Let $n_j = j(j+1)/2$ for $j \ge 1$.
  Note that $$ \nu_{>n_j} = \delta_{N^{-1} \{0,1,\ldots,p-1\} } * \cdots * \delta_{N^{-j} \{0,1,\ldots,p-1\} } * \delta_{N^{-(j+1)} \{1,2,\ldots,p-1,b_{j+1}\} } * \cdots.$$
Then we have that
$\{\nu_{>n_j}\}$ converges weakly to
$$\nu = \frac{p-1}{p} \mu_{N,\{0,1,\ldots,p-1\}} + \frac{1}{p} \delta_{\{n_0+x_0\}} * \mu_{N,\{0,1,\ldots,p-1\}}, $$
where $$\mu_{N,\{0,1,\ldots,p-1\}} = \delta_{N^{-1} \{0,1,\ldots,p-1\} } *  \delta_{N^{-2} \{0,1,\ldots,p-1\} } * \cdots * \delta_{N^{-k} \{0,1,\ldots,p-1\} } * \cdots.$$
Note that $x_0 \not \in \mathrm{spt}(\mu_{N,\{0,1,\ldots,p-1\}})$.
Thus there exists a neighborhood $E$ of $n_0+x_0$ such that $\nu(E)>0$ and $\nu(E+k)=0$ for all $k \in \Z \setminus\{0\}$.
By Theorem \ref{theorem-IPZS-empty-1}, we have that $\mathcal{Z}(\nu) = \emptyset$.
It follows from Theorem \ref{main-spectrality} that the infinite convolution $\mu$ defined in \eqref{example-2-infinite-convolution} is a spectral measure.
\end{example}

\section*{Acknowledgements}
The authors would like to thank the referees for his/her valuable comments and suggestions.
The first author is supported by NSFC No. 12071148, 11971079,  Science and Technology Commission of Shanghai Municipality (STCSM) No.~22DZ2229014.
The second author is supported by  Science and Technology Commission of Shanghai Municipality (STCSM) No.~22DZ2229014.
The third author is supported by Fundamental Research Funds for the Central Universities No.~YBNLTS2023-016.

\end{document}